\documentclass[12pt]{amsart}
\usepackage[T1]{fontenc}
\usepackage[a4paper,top=4cm, bottom=5cm, left=2.75cm, right=2.75cm]{geometry}
\usepackage[utf8]{inputenc}
\usepackage{amsthm,amsmath}  
\usepackage{latexsym,amssymb}
\usepackage{stmaryrd}
\usepackage[shortlabels]{enumitem}
\usepackage{booktabs}
\usepackage[hidelinks]{hyperref}
\usepackage{orcidlink}
\usepackage{color}

\renewcommand{\P}{\mathbb{P}}
\newcommand{\N}{\mathbb{N}}
\newcommand{\A}{\mathbb{A}}
\newcommand{\Z}{\mathbb{Z}}

\newcommand{\mA}{\mathcal{A}}
\newcommand{\mB}{\mathcal{B}}
\newcommand{\mC}{\mathcal{C}}
\newcommand{\mG}{\mathcal{G}}
\newcommand{\mS}{\mathcal{S}}
\newcommand{\mT}{\mathcal{T}}

\newcommand{\bfa}{\mathbf{a}}
\newcommand{\bfy}{\mathbf{y}}
\newcommand{\bfz}{\mathbf{z}}
\newcommand{\bfs}{\mathbf{s}}
\newcommand{\bx}{\mathbf{x}}
\newcommand{\bt}{\mathbf{t}}

\newcommand{\type}{{\rm type}}
\newcommand{\MSG}[1]{{\rm MSG (#1)}}

\newcommand{\Ap}{{\rm Ap}}
\newcommand{\AP}{{\rm AP}}
\newcommand{\ini}{{\rm in}}
\newcommand{\reg}[1]{{\rm reg}(#1)}

\def\acm{arithmetically Cohen-Macaulay}
\def\cm{Cohen-Macaulay}
\def\iff{if and only if}

\theoremstyle{plain}
\newtheorem{theorem}{Theorem}[section]
\newtheorem{lemma}[theorem]{Lemma}
\newtheorem{corollary}[theorem]{Corollary}
\newtheorem{proposition}[theorem]{Proposition}

\newtheorem{prob}[theorem]{Open problem}

\theoremstyle{definition}
\newtheorem{definition}[theorem]{Definition}
\newtheorem{remark}[theorem]{Remark}
\newtheorem{note}[theorem]{Note}

\newtheorem{example}{Example}

\title{Projective Cohen-Macaulay monomial curves and their affine charts}

\author[I. García-Marco]{Ignacio García-Marco\,\orcidlink{0000-0003-4993-7577}}
 \address{Instituto de Matem\'aticas y Aplicaciones (IMAULL), Secci\'on de Matem\'aticas, Facultad de
Ciencias, Universidad de La Laguna, 38200, La Laguna, Spain}
 \email{iggarcia@ull.edu.es}

\author[P. Gimenez]{Philippe Gimenez\,\orcidlink{0000-0002-5436-9837}}
 \address{Instituto de Investigaci\'on en Matem\'aticas de la Universidad de Valladolid (IMUVA), Universidad de Valladolid, 47011 Valladolid, Spain.}
 \email{pgimenez@uva.es}

 \author[M. Gonz\'alez-S\'anchez]{Mario Gonz\'alez-S\'anchez\,\orcidlink{0000-0002-6458-7547}}
 \address{Instituto de Investigaci\'on en Matem\'aticas de la Universidad de Valladolid (IMUVA), Universidad de Valladolid, 47011 Valladolid, Spain.}
 \email{mario.gonzalez.sanchez@uva.es}
 
\thanks{
This work was supported in part by the grant PID2022-137283NB-C22 funded by MICIU/AEI/ 10.13039/501100011033 and by ERDF/EU.
Part of this research was done during a visit of M.G. and Ph.G. at Universidad de La Laguna funded by {\it Plan de Incentivaci\'on de la actividad investigadora ULL}. 
The third author thanks financial support from European Social Fund, {\it Programa Operativo de Castilla y Le\'on}, and {\it Consejer\'ia de Educaci\'on de la Junta de Castilla y Le\'on.}.
}



\begin{document}

\begin{abstract}
In this paper, we explore when the Betti numbers of the coordinate rings of a projective monomial curve and one of its affine charts are identical. Given an infinite field $k$ and a sequence of relatively prime integers $a_0 = 0 < a_1 < \cdots < a_n = d$, we consider the projective monomial curve $\mathcal{C}\subset\mathbb{P}_k^{\,n}$ of degree $d$ parametrically defined by $x_i = u^{a_i}v^{d-a_i}$ for all $i \in \{0,\ldots,n\}$ and its  coordinate ring $k[\mathcal{C}]$. The curve $\mathcal{C}_1 \subset \mathbb A_k^n$ with parametric equations $x_i = t^{a_i}$ for $i \in \{1,\ldots,n\}$ is an affine chart of $\mathcal{C}$ and we denote by $k[\mathcal{C}_1]$ its coordinate ring. 
The main contribution of this paper is the introduction of a novel (Gröbner-free) combinatorial criterion that provides a sufficient condition for the equality of the Betti numbers of $k[\mathcal{C}]$ and $k[\mathcal{C}_1]$. Leveraging this criterion, we identify infinite families of projective curves satisfying this property. Also, we use our results to study the so-called shifted family of monomial curves, i.e., the family of curves associated to the sequences $j+a_1 < \cdots < j+a_n$ for different values of $j \in \mathbb N$. In this context, Vu proved that for large enough values of $j$, one has an equality between the Betti numbers of the corresponding affine and projective curves. Using our results, we improve Vu's upper bound for the least value of $j$ such that this occurs.
\end{abstract}

\maketitle

\section*{Introduction}

Let $k$ be an infinite field, and $k[\bx] := k[x_1,\ldots,x_n]$ and $k[\bt] := k[t_1,\ldots,t_m]$ be two
polynomial rings over $k$.
Given  $\mB = \{b_1,\ldots,b_n\} \subset \N^m$, a set of nonzero vectors, each element $b_i = (b_{i1},\ldots,b_{im}) \in \N^m$ corresponds to the monomial $\bt^{b_i} := t_1^{b_{i1}}\cdots t_m^{b_{im}} \in k[\bt]$.
The affine toric variety $X_{\mB} \subset \mathbb{A}_k^n$ determined by $\mB$ is the Zariski closure of the set given parametrically by $x_i = u_1^{b_{i1}}\cdots u_m^{b_{im}}$ for all $i = 1,\ldots,n$. Consider 
\[ \mS_\mB := \langle b_1,\ldots,b_n\rangle = \{ \alpha_1 b_1 + \cdots + \alpha_n b_n \, \vert \, \alpha_1,\ldots,\alpha_n \in \N\} \subset \N^m\,, \] 
the affine monoid spanned by $\mB$. The toric ideal determined by $\mB$ is the kernel $I_{\mB}$ of the homomorphism of $k$-algebras
$\varphi_\mB: k[\bx] \longrightarrow k[\bt]$ induced by $x_i \mapsto \bt^{b_i}$. Since $k$ is infinite, one has that $I_\mB$ is the vanishing ideal of $X_\mB$; see, e.g., \cite[Cor.~8.4.13]{Villa} and,  hence, the coordinate ring of $X_{\mB}$ is (isomorphic to) the monomial algebra $k[\mS_\mB] := {\rm Im}(\varphi_\mB) \simeq k[\bx]/I_\mB$. 
The ideal $I_\mB$ is an $\mS_\mB$-homogeneous binomial ideal, i.e., if one sets the $\mS_\mB$-degree of a monomial $\bx^{\alpha} \in k[\bx]$ as ${\rm deg}_{\mS_\mB}(\bx^{\alpha}) :=
\alpha_1 b_1 + \cdots + \alpha_n b_n \in \mS_\mB$, then $I_{\mB}$ is generated by $\mS_{\mB}$-homogeneous binomials. 
One can thus consider a minimal $\mS_\mB$-graded free resolution of $k[\mS_\mB]$ as $\mS_\mB$-graded $k[\bx]$-module,
\[ \mathcal F: 0 \longrightarrow F_p \longrightarrow \cdots  \longrightarrow F_0 \longrightarrow k[\mS_\mB] \longrightarrow 0\,. \]
The projective dimension of $k[\mS_\mB]$ is ${\rm pd}(k[\mS_\mB]) = {\rm max}\{i \, \vert \, F_i \neq 0\}$.
The $i$-th Betti number of $k[\mS_\mB]$ is the rank of the free module $F_i$, i.e., $\beta_i(k[\mS_\mB]) = {\rm rank}(F_i)$; and the Betti sequence of $k[\mS_\mB]$ is $( \beta_i(k[\mS_\mB]) \, ; \, 0 \leq i  \leq {\rm pd}(k[\mS_\mB]))$.  When the Krull dimension of $k[\mS_\mB]$ coincides with its depth as $k[\bx]$-module, the ring $k[\mS_\mB]$ is said to be Cohen-Macaulay. By the Auslander-Buchsbaum formula, this is equivalent to ${\rm pd}(k[\mS_\mB]) = n - {\rm dim}(k[\mS_\mB])$. When $k[\mS_\mB]$ is Cohen-Macaulay, its (Cohen-Macaulay) type is the rank of the last nonzero module in the resolution, i.e.,  ${\type}(k[\mS_\mB]) := \beta_p(k[\mS_\mB])$ where $p = {\rm pd}(k[\mS_\mB])$. One says that $k[\mS_\mB]$ is Gorenstein if it is \cm{} of type $1$. Recall that if $I_\mB$ is a complete intersection, then $k[\mS_\mB]$ is Gorenstein.
\newline

Now consider an integer $d>0$ and  a sequence $a_0 = 0 < a_1 < \cdots < a_n = d$ of relatively prime integers, i.e., $\gcd(a_1,\ldots,a_n)=1$.  Denote by $\mC$ the projective monomial curve $\mC\subset\P_k^{\,n}$ of degree $d$ parametrically defined by $x_i = u^{a_i}v^{d-a_i}$ for all $i \in \{0,\ldots,n\}$, i.e., $\mC$ is the Zariski closure of \[ \{(u^{a_0}v^{d-a_0}:\cdots : u^{a_i} v^{d-a_i}: \cdots : u^{a_n} v^{d-a_n}) \in \P_k^{\,n} \, \vert \, (u:v) \in \P_k^1\} . \] 
Taking $\mA = \{\bfa_0,\ldots,\bfa_n\} \subset \N^2$ with $\bfa_i = (a_i,d-a_i)$ for all $i=0,\ldots,n$, one has that $I_\mA$ is the vanishing ideal of $\mC$, and the coordinate ring of $\mC$ is the two-dimensional ring $k[\mC] = k[x_0,\ldots,x_n]/I_\mA \simeq k[\mS_\mA]$, where $\mS_\mA$ denotes the monoid spanned by $\mA$.  The projective monomial curve $\mC$ is said to be \acm{} 
(resp. Gorenstein) if the ring $k[\mC]$ is Cohen-Macaulay (resp. Gorenstein).
\newline

The projective curve $\mC$ has two affine charts,  $\mC_1 =\{(u^{a_1},\ldots,u^{a_n}) \in \A_k^n \, \vert \, u \in k\}$ and $\mC_2 = \{(v^{d-a_0}, v^{d-a_1},\ldots,v^{d-a_{n-1}}) \in \A_k^n \, \vert \, v \in k\}$,
associated to the sequences $a_1 < \cdots < a_n$ and  $d - a_{n-1} < \cdots < d - a_1 < d - a_0$, respectively. The second sequence is sometimes called the {\it dual} of the first one. Denote by $\mS_1 := \mS_{\mA_1}$ the submonoid of $\N$ spanned by $\mA_1$. Since $a_1,\ldots,a_n$ are relatively prime, then $\N \setminus \mS_1$ is finite and $\mS_1$ is called a {\it numerical semigroup}. The vanishing ideal of $\mC_1$ is $I_{\mA_1} \subset k[x_1,\ldots,x_n]$, and hence, its coordinate ring is the one-dimensional ring $k[\mC_1] = k[x_1,\ldots,x_n]/I_{\mA_1} \simeq k[\mS_1]$. Moreover, $I_\mA$ is the homogenization of $I_{\mA_1}$ with respect to the variable $x_0$.
Similarly, denoting by $\mS_2 := \mS_{\mA_2}$ the numerical semigroup generated by $\mA_2 := \{d-a_0,d-a_1,\ldots,d-a_{n-1}\}$, the vanishing ideal of $\mC_2$ is $I_{\mA_2} \subset k[x_0,\ldots,x_{n-1}]$, its coordinate ring is $k[\mC_2] = k[x_0,\ldots,x_{n-1}]/I_{\mA_2} \simeq k[\mS_2]$, and $I_\mA$ is the homogenization of $I_{\mA_2}$ with respect to $x_n$. 
\newline

One has that $\beta_i(k[\mC]) \geq \beta_i(k[\mC_1])$ for all $i$, and the goal of this work is to understand when the Betti sequences of $k[\mC]$ and $k[\mC_1]$ coincide. A  necessary condition is that $k[\mC]$ is Cohen-Macaulay. Indeed, affine monomial curves are always \acm{} while projective ones may be \acm{} or not, and ${\rm pd}(k[\mC]) = {\rm pd}(k[\mC_1])$ if and only if $\mC$ is \acm. The Cohen-Macaulay condition for $k[\mC]$ is thus necessary for the two Betti sequences to have the same length (and hence it is necessary for the two Betti sequences to coincide). In Theorem~\ref{thm:BettiPosets}, which is the main result of this work, we provide a combinatorial sufficient condition for having equality between the Betti sequences of $k[\mC]$ and $k[\mC_1]$ by means of the poset structures induced by $\mS$ and $\mS_1$ on the Apery sets of both $\mS$ and $\mS_1$. 
In Propositions \ref{pr:aritm} and \ref{pr:quasiaritm}, we use our main result to provide explicit families of curves where $\beta_i(k[\mC]) = \beta_i(k[\mC_1])$ for all $i$. In Section \ref{sec:vu}, we apply our results to study the shifted family of monomial curves, i.e., the family of curves associated to the sequences $j+a_1 < \cdots < j+a_n$ parametrized by $j \in \mathbb N$. In this setting, Vu proved in \cite{Vu} that the Betti numbers in the shifted family become periodic in $j$ for $j > N$ for an integer $N$ explicitly given. A key step in his argument is to prove that for $j > N$ one has equality between the Betti numbers of the  affine and projective curves. Using our results, we substantially improve this latter bound in Theorem~\ref{thm:mejoravu}. Finally, we show in Section \ref{sec:gorenstein} how to construct arithmetically Gorenstein projective curves from a symmetric numerical semigroup (Thm.~\ref{th:gorenstein}).
\newline

There are several sources of  motivation for this work. In \cite{GSS} the authors completely describe the minimal graded free resolution of affine monomial curves defined by an arithmetic sequence $a_1 < \cdots < a_n$, i.e. $a_i = a_1+(i-1)e$ for some $e\in \Z^+$ such that $\gcd(a_1,e) = 1$. In particular, in \cite[Thm.~4.7]{GSS}, they deduce that the Cohen-Macaulay type of these curves, i.e., the last Betti number in the Betti sequence, is given by the only integer $t \in \{1,\ldots,n-1\}$ such that $t \equiv a_1 - 1 \ ({\rm mod}\ n-1)$. Interestingly, later in \cite[Thm.~2.13]{BGG}, exactly the same result is obtained for projective monomial curves defined by an arithmetic sequence. Moreover, some computer-assisted experiments using \cite{Singular} showed that in some cases the whole Betti sequences coincide, while in other cases they do not. 

\begin{example}\label{ex:intro}
\begin{enumerate}[(a)]
\item \label{ex:intro_a} For $4 < 5  < 6 < 7 < 8$, $\mC$ is \acm{} and the Betti sequences of $k[\mC_1]$ and $k[\mC]$ are both $(1,7,14,11,3)$. 
\item \label{ex:intro_b} For $4 < 5 < 6 < 7 < 8 < 9$, $\mC$ is \acm{} and the Betti sequence of $k[\mC_1]$ is $(1,8,21,25,14,3)$, while the Betti sequence of $k[\mC]$ is $(1,12,25,25,14,3)$. Observe that, in this example, the Cohen-Macaulay type of both $k[\mC]$ and $k[\mC_1]$ is 3, as expected by \cite[Thm.~4.7]{GSS} and \cite[Thm.~2.13]{BGG}, but the two Betti sequences do not coincide. 
\end{enumerate}
\end{example}

On the other hand, Vu proves in \cite{Vu} that given $a_1 < \cdots < a_n$, if one considers the sequence of positive integers $j + a_1 < \cdots < j + a_n$ parametrized by $j \in \N$, then the  Betti numbers of the corresponding affine monomial curves  become eventually periodic. One of the key steps in Vu's argument is to prove that for values of $j$ big enough, the Betti sequences of the affine and projective monomial curves coincide.  Later, Saha, Sengupta and Srivastava obtain in \cite{Sengupta2023} a sufficient condition in terms of Gr\"obner bases to ensure the equality of the Betti sequences.  \newline

Finally, Jafari and Zarzuela provide in \cite{JZ} a combinatorial sufficient condition for having equality between the Betti sequences of $k[\mC_1]$ and of the coordinate ring of its corresponding tangent cone $G(\mC_1)$. More precisely, they define homogeneous numerical semigroups as those whose Apery set with respect to $a_1$ has a graded poset structure with respect to the partial ordering induced by $\mS_1$. In \cite[Thm.~3.17]{JZ}, they prove that the Betti sequences of $k[\mC_1]$ and $G(\mC_1)$ coincide when $\mS_1$ is homogeneous and $G(\mC_1)$ is Cohen-Macaulay. The spirit of our approach in this paper is similar although the problems addressed are different as the next examples show.

\begin{example}
\begin{enumerate}[(a)]
    \item \cite[Ex.~5.5]{EjShibuta} For the sequence $3m < 3m+1 < 6m+3$ with $m \geq 2$, one has that $\beta_1(k[\mC_1]) = 2$ and  $\beta_1(G(\mC_1)) = m+2$. Indeed,
$I_\mA = (x_2^3 - x_1 x_3 x_4, x_1^{2m+1} - x_3^m x_4^{m+1})$, and the Betti sequences of $k[\mC_1]$ and $k[\mC]$ are both $(1,2,1)$, while $\beta_1(G(\mC_1)) = m+2 \geq 4$. 
Computational experiments in \cite{Stamate} suggest that in this example, the Betti sequence of $G(\mC_1)$ is $(1, m + 2, 2m, m-1)$. 
\item \cite[Cor.~3.2]{BGG} Numerical semigroups defined by a generalized arithmetic sequence, i.e., $\mS_1 = \langle a, ha+e,\ldots,ha+(n-1)e\rangle$ for positive integers $h, a, e, n$, always satisfy that the Betti sequences of $k[\mC_1]$ and $G(\mC_1)$ are equal, and hence, $G(\mC_1)$ is Cohen-Macaulay (see \cite[Cor.~3.10]{JZ}). However, whenever $h > 1$ and $n \geq 3$, $k[\mC]$ is not Cohen-Macaulay, so the Betti sequences of $k[\mC_1]$ and $k[\mC]$ are different. When $h=1$, i.e., when the sequence is arithmetic, then $k[\mC]$ is  Cohen-Macaulay but the Betti sequences may also be different as Example \ref{ex:intro}\ref{ex:intro_b} shows.
\end{enumerate}
\end{example}

\subsection*{Notations}
For a numerical semigroup $\mS\subset \N$, we denote by $\MSG{\mS}$ its minimal system of generators. 
If $\bx^\alpha \in k[x_1,\ldots,x_n]$ is a monomial, where $\alpha=(\alpha_1,\ldots,\alpha_n) \in \N^n$, we denote its degree by $|\alpha| = \sum_{i=1}^n \alpha_i$. For a finite set $A$, $|A|$ is its cardinality. 
For a finite subset $A \subset \N$ and a natural number $s\in \N$, $s\geq 1$, the $s$-fold iterated sumset of $A$ is $sA := \{a_1+\dots+a_s \mid a_1,\ldots,a_s \in A\}$. We also denote $0A:=\{0\}$.

Along the paper, we consider $>$, the degree reverse lexicographical order (degrevlex) on $k[x_1,\ldots,x_n]$ with $x_1>x_2>\dots>x_n$. This means that, for $\bx^\alpha,\bx^\beta \in k[x_1,\ldots,x_n]$, $\bx^\alpha<\bx^\beta$ if either $|\alpha|<|\beta|$, or $|\alpha|=|\beta|$ and the last nonzero coordinate of $\beta-\alpha$ is negative. Note that our results hold if one chooses a degree reverse lexicographical order on $k[x_1,\ldots,x_n]$ with $x_n<x_i$ for all $i\in \{1,\ldots,n-1\}$. 
For $f\in k[x_1,\ldots,x_n]$, $\ini_>(f)$ denotes the initial term of $f$, and given an ideal $I\subset k[x_1,\ldots,x_n]$, $\ini_>(I)$ denotes the initial ideal of $I$, both with respect to $>$.

\section{Apery sets and their poset structure} \label{sec:AffineProjMonomialCurves}

Let $d \in \Z^+$ and $a_0 = 0 < a_1 < \cdots < a_n = d$ be a sequence of relatively prime integers. For each $i=0,\ldots,n$, set $\bfa_i := (a_i,d-a_i) \in \N^2$, and consider the three sets $\mA_1 = \{a_1,\ldots,a_n\}$, $\mA_2 = \{d,d-a_{1},\ldots,d-a_{n-1}\}$ and $\mA = \{\bfa_0,\ldots,\bfa_n\} \subset \N^2$.
We denote by $\mC \subset \P_k^{\,n}$ the projective monomial curve parametrized by $\mA$ as defined in the introduction, and by $\mC_1$ and $\mC_2$ its affine charts, i.e., the affine monomial curves given by $\mA_1$ and $\mA_2$, respectively. We denote the vanishing ideal of $\mC_i$ by $I_{\mA_i}$ for $i=1,2$ and the vanishing ideal of $\mC$ by $I_\mA$.
Consider $\mS_1$ and $\mS_2$ the numerical semigroups generated by $\mA_1$ and $\mA_2$ respectively, and $\mS$ the monoid spanned by $\mA$ that we call the homogenization of $\mS_1$ (with respect to $d$). \newline

As already mentioned, $k[\mC_1]$ and $k[\mC_2]$ are always \cm{}, while $k[\mC]$ can be \cm{} or not. There are many ways to determine when a projective monomial curve is \acm; see, e.g.,  \cite[Cor.~4.2]{CG}, \cite[Lem.~4.3, Thm.~4.6]{CN}, \cite[Thm.~2.6]{GSW} or \cite[Thm.~2.2]{HerzogStamate2019}. We recall some of them in Proposition~\ref{prop:caractCM}, but let us previously introduce the notion of Apery set since it is involved in some of those charaterizations. 
For $i = 1,2$, the {\it Apery set of $\mS_i$ with respect to $d$} is
\[{\Ap}_i := \{ y \in \mS_i \, \vert \, y - d \notin \mS_i\}\,.\]
Since $\gcd(\mA_1) = 1$, we know that ${\Ap}_i$ is a complete set of residues modulo $d$, i.e., 
${\Ap}_1 =  \{r_0 = 0,r_1,\ldots,r_{d-1}\}$ and ${\Ap}_2 = \{t_0 = 0,t_1,\ldots,t_{d-1}\}$ for some positive integers $r_i$ and $t_i$ such that $r_i \equiv t_i \equiv i \pmod{d}$ for all $i=1,\ldots,d-1$. 
One can also define the {\it Apery set of $\mS$} as 
\[\AP_{\mS} := \{\bfy \in \mS \, \vert \, \bfy - \bfa_0 \notin \mS, \bfy - \bfa_{n} \notin \mS\}\,.\]
This set is finite and has at least $d$ elements by \cite[Lem.~2.5]{GG23}.

\begin{proposition}\label{prop:caractCM}
The following statements are equivalent:
 \begin{enumerate}[(a)] 
 \item\label{prop:caractCMa} $\mC$ is \acm{}.
 \item\label{prop:caractCMb} $\AP_{\mS}$ has exactly $d$ elements.
 \item\label{prop:caractCMc} $\AP_{\mS} = \{(0,0)\} \cup \{(r_i, t_{d - i}) \, \vert \, 1 \leq i < d\}$.
 \item\label{prop:caractCMd} For all $i=1,\ldots,d-1$, $(r_i,t_{d-i}) \in \mS$. In other words, if $q_1 \in \Ap_1,\ q_2 \in \Ap_2$ and $q_1 + q_2 \equiv 0 \pmod{d}$, then $(q_1,q_2) \in \mS$.
 \item\label{prop:caractCMe} If $\bfs \in \Z^2$ satisfies $\bfs+\bfa_0\in\mS$ and $\bfs+\bfa_{n}\in\mS$, then $\bfs \in \mS$.
 \item\label{prop:caractCMf} The variable $x_n$ does not divide any minimal generator of $\ini_>(I_{\mA_1})$, the initial ideal of $I_{\mA_1}$ for the degrevlex order in $k[x_1,\ldots,x_n]$ with $x_1>\dots>x_n$.
 \end{enumerate}
\end{proposition}

For $i=1,2$, one can consider the order relation $\leq_i$ in $\mS_i$ given by $y\leq_i z \Longleftrightarrow z-y\in \mS_i$. Similarly, in $\mS$ one can consider the order relation $\leq_\mS$ defined by $\bfy \leq_\mS \bfz \Longleftrightarrow \bfz-\bfy \in \mS$. 
In order to compare $\beta_i(k[\mC])$ and $\beta_i(k[\mC_1])$ for all $i$, we will relate in Theorem~\ref{thm:BettiPosets} the posets $(\Ap_1,\leq_1)$ and $(\AP_\mS,\leq_\mS)$, i.e., the Apery sets $\Ap_1 \subset \mS_1$ and $\AP_\mS \subset \mS$ with the natural poset structure they inherit from $(\mS_1,\leq_1)$ and $(\mS,\leq_\mS)$, respectively. \newline

Since $\mS \subset \mS_1 \times \mS_2$, it follows that if $(y_1,y_2) \leq_\mS (z_1,z_2)$, then $y_i \leq_i z_i$ for $i=1,2$. Using Proposition~\ref{prop:caractCM}, one can prove that the poset structure of $(\AP_{\mS}, \leq_{\mS})$ is completely determined by those of $(\Ap_1,\leq_1)$ and $(\Ap_2,\leq_2)$ when $\mC$ is \acm{}.

\begin{proposition} \label{pr:ordeninterseccion}  
If $\mC$ is \acm{}, then for all $(y_1,y_2), (z_1,z_2) \in \AP_\mS$, 
\[ (y_1,y_2) \leq_{\mS} (z_1,z_2) \ \Longleftrightarrow \ y_1 \leq_1 z_1 \text{ and } y_2 \leq_2 z_2. \]
\end{proposition}

\begin{proof}
As observed before stating the proposition, $(\Rightarrow)$ always holds. Let us prove $(\Leftarrow)$ when $\mC$ is \acm. Since $(y_1,y_2),(z_1,z_2) \in \AP_\mS$, one has that $y_1,z_1\in \Ap_1$, $y_2,z_2 \in \Ap_2$ by Proposition~\ref{prop:caractCM}~\ref{prop:caractCMc}, and $y_1+y_2 \equiv z_1+z_2 \equiv 0 \pmod{d}$. Assume that $y_1\leq_1 z_1$ and $y_2 \leq_2 z_2$, then $s_1:=z_1-y_1 \in \mS_1$ and $s_2:=z_2-y_2 \in \mS_2$. Moreover, $s_i \in \Ap_i$ for $i=1,2$; otherwise, $z_i \notin \Ap_i$. Since $s_1+s_2 = z_1+z_2-y_1-y_2 \equiv 0 \pmod{d}$, then $(s_1,s_2) \in \mS$ by Proposition~\ref{prop:caractCM}~\ref{prop:caractCMd}, and we are done.
\end{proof}

Let us recall now some notions about posets that will be used in the sequel for the posets $(\Ap_1,\leq_1)$, $(\Ap_2,\leq_2)$ and $(\AP_\mS,\leq_\mS)$.

\begin{definition}
Let $(P,\leq)$ be a finite poset.
\begin{enumerate}[(a)]
    \item For $y,z\in P$, we say that {\it $z$ covers $y$}, and denote it by $y\prec z$, if $y<z$ and there is no $w\in P$ such that $y<w<z$.
    \item We say that $P$ is \textit{graded} if there exists a function $\rho:P\rightarrow \N$, called  {\it rank function}, 
    such that $\rho(z)=\rho(y)+1$ whenever $y \prec z$. 
\end{enumerate}
\end{definition}

The following result shows that the poset $(\AP_\mS,\leq_\mS)$ is always graded while $(\Ap_1,\leq_1)$ may be graded or not. Observe that, since $(\Ap_1,\leq_1)$ has a minimum element which is $0$, whenever it is graded, the corresponding rank function is completely determined by the value of the rank function at $0$ that we will fix to be $0$. In the following proposition, we characterize the covering relation in $\Ap_1$ and in $\AP_\mS$, and describe the rank functions of $(\AP_\mS,\leq_\mS)$, and of $(\Ap_1,\leq_1)$ when it is graded.

\begin{proposition} \label{prop:CoverRel&RankFunct}
\begin{enumerate}[label=(a.\arabic*)]
    \item\label{prop:CoverRel_a1}  For all $y,z\in \Ap_1$, $y \prec_1 z$ $\Longleftrightarrow$ $z=y+a_i$ for some $a_i \in \MSG{\mS_1}\setminus\{a_n\}$.
    \item\label{prop:CoverRel_a2} $\Ap_1$ is graded \iff{}, for all $y\in \Ap_1$, every way of writing $y$ as a sum of elements in $\MSG{\mS_1}$ has the same number of summands. When it is graded, the rank function $\rho_1:\Ap_1\rightarrow \N$ is given by that number of summands.
\end{enumerate}
\begin{enumerate}[label=(b.\arabic*)]
    \item\label{prop:CoverRel_b1}  For all $\bfy=(y_1,y_2),\bfz=(z_1,z_2)\in \AP_\mS$, $\bfy \prec_\mS \bfz$ $\Longleftrightarrow$ $\bfz=\bfy+\bfa_i$ for some $i \in \{1,\ldots,n-1\}$.
    \item\label{prop:CoverRel_b2}  $\AP_\mS$ is graded by the rank function $\rho:\AP_\mS \rightarrow \N$ defined by $\rho(y_1,y_2):= (y_1+y_2)/d$.
\end{enumerate} 
\end{proposition}

\begin{proof}
In \ref{prop:CoverRel_a1} and \ref{prop:CoverRel_b1}, $(\Leftarrow)$ is trivial. Let us prove $(\Rightarrow)$. 
\begin{enumerate}
    \item[\ref{prop:CoverRel_a1}] Consider $y,z\in \Ap_1$ such that $y\prec_1 z$. Since $z-y\in \mS_1$, there exists $\alpha=(\alpha_1,\ldots,\alpha_n) \in \N^n$ such that $z=y+\sum_{i=1}^n \alpha_i a_i$, and $\alpha_n=0$ because $z\in \Ap_1$. If $|\alpha| >1$, then there exists $j \in \{1,\ldots,n-1\}$ such that $\alpha_j \neq 0$ and $y+a_j\neq z$. Thus, $y+a_j \in \Ap_1$ because $z\in \Ap_1$, and $y <_1 y+a_j <_1 z$, a contradiction because $y\prec_1 z$, so $|\alpha| = 1$.
    \item[\ref{prop:CoverRel_b1}] Consider $\bfy,\bfz \in \AP_\mS$ such that $\bfy \prec_\mS \bfz$. Since $\bfz-\bfy \in \mS$, there exists $\alpha=(\alpha_0,\ldots,\alpha_n) \in \N^{n+1}$ such that $\bfz-\bfy = \sum_{i=0}^n \alpha_i \bfa_i$, and $\alpha_0=\alpha_n = 0$ because $\bfz \in \AP_\mS$. Again, if $|\alpha| = \sum_{i=1}^{n-1} \alpha_i > 1$, we can choose any $\alpha_j \neq 0$ and get $\bfy <_\mS \bfy+\bfa_j <_\mS \bfz$ (one has that $\bfy+\bfa_j \in \AP_\mS$ because $\bfz \in \AP_\mS$), a contradiction because $\bfy \prec_\mS \bfz$.
\end{enumerate}
Now \ref{prop:CoverRel_a2} and \ref{prop:CoverRel_b2} are direct consequences of \ref{prop:CoverRel_a1} and \ref{prop:CoverRel_b1}, respectively.
\end{proof}

\begin{remark} \label{rem:Fiber(1)}
By Proposition~\ref{prop:CoverRel&RankFunct}~\ref{prop:CoverRel_b2}, the fiber of $1$ under the rank function $\rho$ is $\rho^{-1}(1) = \{\bfa_i: 1\leq i \leq n-1\}$, and hence $|\rho^{-1}(1)| = n-1$.
On the other hand, when $\Ap_1$ is graded, the fiber of $1$ under $\rho_1$ is $\rho_1^{-1}(1) = \MSG{\mS_1} \setminus \{a_n\}$ by Proposition~\ref{prop:CoverRel&RankFunct}~\ref{prop:CoverRel_a1}.
\end{remark}

Set $\mA_1' := \MSG{\mS_1} \setminus \{a_n\}$ and $\Ap_1^{(s)} := \Ap_1 \cap s\mA_1'$ for each $s\in \N$. 
Since $\Ap_1$ is finite, consider $N:=\max\{s\in \N: \Ap_1^{(s)} \neq \emptyset\} \in \N$.
As a direct consequence of Proposition~\ref{prop:CoverRel&RankFunct}~\ref{prop:CoverRel_a2}, we get a characterization of the graded property for $(\Ap_1,\leq_1)$. 

\begin{corollary}
$(\Ap_1,\leq_1)$ is graded if and only if $\sum_{s=0}^N |\Ap_1^{(s)}| = d$.
\end{corollary}

\section{Betti numbers of affine and projective monomial curves} \label{sec:BettiNumbers}

Recall that $I_{\mA_1} \subset k[x_1,\ldots,x_n]$ and $I_\mA \subset k[x_0,\ldots,x_n]$ are
the vanishing ideals of $\mC_1$ and $\mC$ respectively. 
When $\mC$ is arithmetically Cohen-Macaulay, ${\rm pd}(k[\mC]) = {\rm pd}(k[\mC_1])$.
Moreover, by Proposition~\ref{prop:APS_Ap1_isomorphic}~\ref{prop:caractCMb}, in this case, one has that $|\AP_\mS|=|\Ap_1|=d$. The main result in this section is Theorem~\ref{thm:BettiPosets} where we give a sufficient condition in terms of the poset structures of the Apery sets $\Ap_1$ and $\AP_\mS$ for the Betti sequences of $k[\mC_1]$ and $k[\mC]$ to coincide. We postpone the proof after Propositions \ref{prop:APS_Ap1_isomorphic} and \ref{prop:Ap1graded_Sengupta}.

\begin{theorem} \label{thm:BettiPosets}
If $(\AP_\mS, \leq_{\mS}) \simeq (\Ap_1, \leq_1)$, then $\beta_i(k[\mC]) = \beta_i(k[\mC_1])$ for all $i$.
\end{theorem}

Note that the converse of this result is wrong as the following example shows.

\begin{example} \label{ex:convthm2.1}  
For the sequence $1<2<4<8$, one can check using, e.g., \cite{Singular}, that both $k[\mC_1]$ and $k[\mC]$ are complete intersections with Betti sequence $(1,3,3,1)$. However, the posets $(\Ap_1,\leq_1)$ and $(\AP_\mS,\leq_\mS)$ are not isomorphic since $\leq_1$ is a total order on $\Ap_1$, while $\leq_\mS$ is not. 
More generally, for $a_1 = 1 < a_2 < \cdots < a_n = d$ with $a_i$ a divisor of $a_{i+1}$ for all $i \in \{1,\ldots,n-1\}$, one has that both $k[\mC_1]$ and $k[\mC]$ are complete intersections; see \cite[Theorem 5.3]{BGM}. Thus, both Betti sequences are $(1,\binom{n-1}{1},\ldots,\binom{n-1}{i},\ldots, \binom{n-1}{n-2}, 1)$. However, again the posets $(\Ap_1,\leq_1)$ and $(\AP_\mS,\leq_\mS)$ are not isomorphic since $\leq_1$ is a total order on $\Ap_1$, while $\leq_\mS$ is not.
\end{example}

\begin{proposition} \label{prop:APS_Ap1_isomorphic} 
The following two claims are equivalent:
\begin{enumerate}[(a)]
    \item The posets $(\Ap_1,\leq_1)$ and $(\AP_\mS , \leq_\mS)$ are isomorphic;
    \item\label{prop:APS_Ap1_isomorphic_b} $k[\mC]$ is Cohen-Macaulay, $(\Ap_1,\leq_1)$ is graded, and $\{a_1,\ldots,a_{n-1}\}$ is contained in the minimal system of generators of $\mS_1$.
\end{enumerate}
\end{proposition}

\begin{proof}
\underline{$(a) \Rightarrow (b)$}. 
If $(\AP_\mS, \leq_{\mS}) \simeq (\Ap_1, \leq_1)$, then $\Ap_1$ and $\AP_\mS$ have the same number of elements, and hence $k[\mC]$ is Cohen-Macaulay by Proposition~\ref{prop:caractCM}~\ref{prop:caractCMb}. Moreover, since $(\AP_\mS,\leq_\mS)$ is graded by Proposition~\ref{prop:CoverRel&RankFunct}~\ref{prop:CoverRel_b2}, $(\Ap_1,\leq_1)$ is graded. 
Finally, $|\rho_1^{-1}(1)| =|\rho^{-1}(1)|$ so, by Remark~\ref{rem:Fiber(1)}, $|\MSG{\mS_1} \setminus \{a_n\}| = n-1$, and hence $\{a_1,\ldots,a_{n-1}\} \subset \MSG{\mS_1}$.\\
\underline{$(b) \Rightarrow (a)$}. 
If $k[\mC]$ is Cohen-Macaulay, then $|\AP_\mS| = |\Ap_1|$ by Proposition~\ref{prop:caractCM}~\ref{prop:caractCMb}, and hence the map $\varphi: \AP_\mS \rightarrow \Ap_1$ defined by $\varphi(r_j,t_{d-j}) = r_j$ for all $j=0,\ldots,d-1$, is bijective. Let us prove that it is an isomorphism of posets. 
By Proposition~\ref{pr:ordeninterseccion}, $\varphi$ is an order-preserving map, so one just has to show that $\varphi^{-1}$ is also order-preserving.
Consider $y_1,z_1\in \Ap_1$ such that $y_1 \prec_1 z_1$. Then, there exists $i \in \{1,\ldots,n-1\}$ such that $z_1=y_1+a_i$ by Proposition~\ref{prop:CoverRel&RankFunct}~\ref{prop:CoverRel_a1}.
Moreover, $y_2+d-a_i \geq z_2$ since $z_2\in \Ap_2$ and $y_2+d-a_i \in \mS_2$. Note that $\rho(y_1,y_2) = \rho_1(y_1)$ (and the same holds for $(z_1,z_2)$). This is because if we write $(y_1,y_2) = \sum_{i=1}^{n-1} \alpha_i (a_i,d-a_i)$ for some $\alpha_i \in \N$, then $y_1 = \sum_{i=1}^{n-1} \alpha_i a_i$ is a way of writing $y_1$ as a sum of elements in $\MSG{\mS_1}$, and hence $\rho(y_1,y_2) = \sum_{i=1}^{n-1} \alpha_i = \rho_1(y_1)$ by Proposition~\ref{prop:CoverRel&RankFunct}~\ref{prop:CoverRel_a2} and \ref{prop:CoverRel_b2}. If $z_2<y_2+d-a_i$, then $\rho_1(z_1) = \rho(z_1,z_2)\leq \rho(y_1,y_2) = \rho_1(y_1)$, a contradiction since $y_1 \prec_1 z_1$. Therefore, $y_2+d-a_i=z_2$ and we are done. 
\end{proof}

Note that $\Ap_1$ can be a graded poset even if $(\Ap_1,\leq_1)$ and $(\AP_\mS,\leq_\mS)$ are not isomorphic as the following example shows.

\begin{example}
For the sequence $a_1=5<a_2=11<a_3=13$, the Apery set of the numerical semigroup $\mS_1 = \langle a_1,a_2,a_3\rangle$ is $\Ap_1 = \{0,27,15,16,30,5,32,20,21,22,10,11,25\}$. This Apery set is graded with the  rank function $\rho_1:\mS_1\rightarrow \N$ defined below (see Figure~\ref{fig:rank_fun}):
\begin{itemize}
    \item $\rho_1(0) = 0$,
    \item $\rho_1(5) = \rho(11) = 1$,
    \item $\rho_1(10)=\rho_1(16)=\rho_1(22) = 2$,
    \item $\rho_1(15)=\rho_1(21)=\rho_1(27) = 3$,
    \item $\rho_1(20)=\rho_1(32)=4$,
    \item $\rho_1(25)=5$,
    \item $\rho_1(30)=6$.
\end{itemize}
\begin{figure}
\includegraphics[scale=.6]{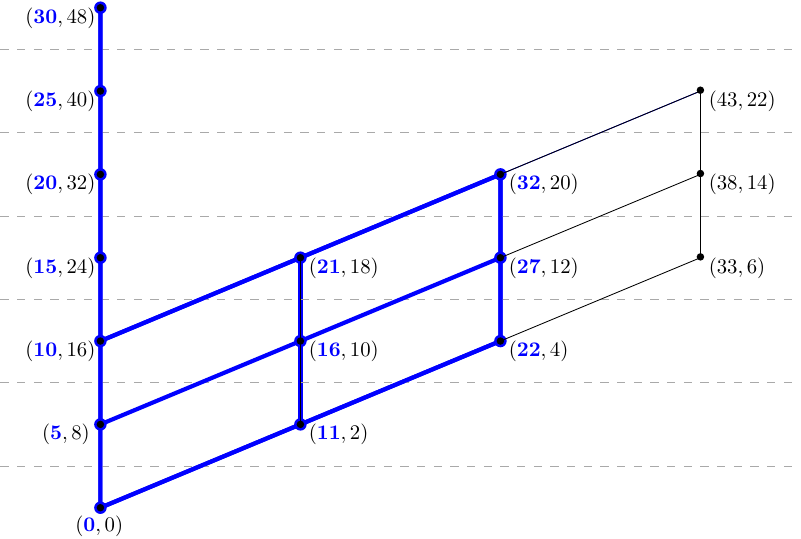}  

\caption{The posets $(\Ap_1,\leq_1)$ (in blue) and $(\AP_\mS,\leq_\mS)$ (in black) for $\mS_1 = \langle 5,11,13\rangle$.}
\label{fig:rank_fun}
\end{figure}
Moreover, since $\AP_\mS$ has 16 elements, $k[\mC]$ is not Cohen-Macaulay, and hence $(\Ap_1,\leq_1)$ and $(\AP_\mS,\leq_\mS)$ are not isomorphic by Proposition~\ref{prop:APS_Ap1_isomorphic}.
\end{example}

We now relate the condition in Proposition~\ref{prop:APS_Ap1_isomorphic} to the criterion in \cite[Thm.~4.1]{Sengupta2023} that uses Gröbner bases.

\begin{proposition} \label{prop:Ap1graded_Sengupta} 
Consider the following two claims:
\begin{enumerate}[(a)]
    \item\label{prop:Ap1graded_Sengupta_a} $(\Ap_1,\leq_1)$ is graded and $\{a_1,\ldots,a_{n-1}\}$ is contained in the minimal system of generators of~$\mS_1$.
    \item\label{prop:Ap1graded_Sengupta_b} The variable $x_n$ appears in every non-homogeneous binomial of $\mG_>$, the reduced Gr\"{o}bner basis of $I_{\mA_1}$ with respect to the degrevlex order with $x_1>x_2>\dots>x_n$. 
\end{enumerate}
Then $\ref{prop:Ap1graded_Sengupta_b} \Rightarrow \ref{prop:Ap1graded_Sengupta_a}$, and $\ref{prop:Ap1graded_Sengupta_a} \Rightarrow \ref{prop:Ap1graded_Sengupta_b}$ holds if $k[\mC]$ is \cm.
\end{proposition}

\begin{proof}
\underline{$(a) \Rightarrow (b)$ when $k[\mC]$ is \cm}. Assume that there exists a non-homo\-ge\-neous binomial $f = \bx^\alpha-\bx^\beta \in \mG_>$ with $\ini_>(f)=\bx^\alpha$ such that $x_n$ does not appear in the support of $f$, i.e. $|\alpha| > |\beta|$ and $\alpha_n=\beta_n=0$, and 
consider $s = \sum_{i=1}^{n-1} \alpha_i a_i = \sum_{i=1}^{n-1} \beta_i a_i \in \mS_1$. Let us prove that $s-a_n \notin \mS_1$. If $s-a_n \in \mS_1$, we can write $s$ as $s = \sum_{i=1}^{n} \gamma_i a_i + a_n$ for some $\gamma=(\gamma_1,\ldots,\gamma_n) \in \N^n$, and consider the binomial $g=\bx^\gamma x_n-\bx^\beta \in I_{\mA_1}$. Note that $\bx^\beta-\bx^\gamma x_n \neq 0$ since $\beta_n = 0$. As $f\in\mG_>$ and $\mG_>$ is reduced, one has that $\bx^\beta \notin \ini_>(I_{\mA_1})$ and hence $\ini_>(g)=\bx^\gamma x_n$. Therefore, $\bx^\gamma x_n \in \ini_> (I_{\mA_1})$ and, by Proposition~\ref{prop:caractCM}~\ref{prop:caractCMf}, $\bx^\gamma \in \ini_> (I_{\mA_1})$.  
The remainder of the division of $\bx^\gamma$ by $\mG_>$ is a monomial $\bx^\delta$ such that $\bx^\delta \notin \ini_> (I_{\mA_1})$, and one has that the binomial $\bx^\beta-\bx^\delta x_n \in I_{\mA_1}$ is the difference of two binomials that do not belong to $\ini_> (I_{\mA_1})$ using again Proposition~\ref{prop:caractCM}~\ref{prop:caractCMf}, a contradiction. Thus, $s-a_n \notin \mS_1$ and hence $s\in \Ap_1$. But  $s = \sum_{i=1}^{n-1} \alpha_i a_i = \sum_{i=1}^{n-1} \beta_i a_i \in \mS_1$ with $|\alpha| > |\beta|$ so if $\{a_1,\ldots,a_{n-1}\} \subset \MSG{\mS_1}$, one gets by Proposition~\ref{prop:CoverRel&RankFunct}~\ref{prop:CoverRel_a2} that $\Ap_1$ is not graded.
\\
\underline{$(b) \Rightarrow (a)$}: If $\{a_1,\ldots,a_{n-1}\} \not\subset \MSG{\mS_1}$, select $i\in \{2,\ldots,n-1\}$ such that $a_i$ is not a minimal generator. Then, there exists $\alpha=(\alpha_1,\ldots,\alpha_{i-1}) \in \N^{i-1}$ with $|\alpha|>2$ such that $x_i-\prod_{j<i} x_j^{\alpha_j}\in I_{\mA_1}$. Note that any set of generators of $I_{\mA_1}$ contains an element of this form. Thus, $\mG_>$ contains a non-homogeneous binomial that does not involve the variable $x_n$, a contradiction, and hence $\{a_1,\ldots,a_{n-1} \} \subset \MSG{\mS_1}$.

If $(\Ap_1,\leq_1)$ is not graded, by Proposition~\ref{prop:CoverRel&RankFunct}~\ref{prop:CoverRel_a2}, there exists
$s\in \Ap_1$ that can be written in two different ways using a different number of minimal generators of $\mS_1$, i.e. $s = \sum_{i=1}^{n-1} \alpha_i a_i = \sum_{i=1}^{n-1} \beta_i a_i$ with $|\alpha|>|\beta|$. Note that $\alpha_n=\beta_n=0$ since $s\in \Ap_1$. We can choose $\beta=(\beta_1,\ldots,\beta_{n-1})$ such that $|\beta|>0$ is the least possible value, and $\alpha=(\alpha_1,\ldots,\alpha_{n-1})$ such that, for this election of $\beta$, $\bx^\alpha$ is the smallest possible monomial for the degree reverse lexicographic order. Then $f=\bx^\alpha-\bx^\beta \in I_{\mA_1}$ and $\ini_>(f)=\bx^\alpha$. Since $\bx^\alpha \in \ini_>(I_{\mA_1})$, there exists a binomial $h=\bx^\lambda-\bx^\mu \in \mG_>$ such that $\bx^\lambda$ divides $\bx^\alpha$. 
Let us see that $h$ is not homogeneous and that the variable $x_n$ is not involved in $h$. 
If $h$ is homogeneous, dividing $\bx^\alpha$ by $h$, we get $\bx^\alpha = \bx^{\alpha-\lambda} (\bx^\lambda-\bx^\mu)+\bx^{\alpha-\lambda+\mu}$. Then, 
$s = \sum_i (\alpha_i-\lambda_i+\mu_i) a_i = \sum_i \alpha_i a_i$ with $|\alpha-\lambda+\mu| = |\alpha|$ and $\bx^{\alpha-\lambda+\mu} < \bx^\alpha$, a contradiction with the choice of $\alpha$, so $h$ is not homogeneous. 
On the other hand, since $\bx^\lambda$ divides $\bx^\alpha$ and $\alpha_n =0$, if $x_n$ appears in $\bx^\lambda-\bx^\mu$, it must be in the support of $\bx^\mu$. 
If we write $\bx^\mu = \bx^{\mu'}x_n$, then $\bx^\alpha = \bx^{\alpha-\lambda} (\bx^\lambda-\bx^{\mu}) + \bx^{\alpha-\lambda+\mu'} x_n$ and hence $s = \sum_i (\alpha_i-\lambda_i+\mu'_i) a_i+a_n $ which is impossible because $s\in \Ap_1$. Therefore, we have found a non-homogeneous binomial $h=\bx^\lambda-\bx^\mu \in \mG_>$ where the variable $x_n$ is not involved, a contradiction. Thus, $(\Ap_1,\leq_1)$ is graded.
\end{proof}

\begin{note}
In our proof of $\ref{prop:Ap1graded_Sengupta_a} \Rightarrow \ref{prop:Ap1graded_Sengupta_b}$, we strongly use that $k[\mC]$ is Cohen-Macaulay but we could not find any non-Cohen-Macaulay example where this implication is wrong.
\end{note}

\begin{proof}[Proof of Theorem~\ref{thm:BettiPosets}]
By Propositions \ref{prop:APS_Ap1_isomorphic} and \ref{prop:Ap1graded_Sengupta}, the Apery posets $(\AP, \leq_{\mS})$ and $(\Ap_1, \leq_1)$ are isomorphic \iff{} the variable $x_n$ appears in every non-homogeneous binomial of $\mG_>$, the reduced Gröbner basis of $I_{\mA_1}$ with respect to the degrevlex order with $x_1>x_2>\dots>x_n$. By applying a recent result of Saha, Sengupta and Srivastava \cite[Thm.~4.1]{Sengupta2023}, our result follows.
\end{proof}

\subsection*{Families of curves where the Betti sequences coincide.}
In Propositions \ref{pr:aritm} and \ref{pr:quasiaritm} below, we provide sequences $a_1<\dots<a_n$ for which the condition in Theorem~\ref{thm:BettiPosets} is satisfied. 
\newline

Let us start with arithmetic sequences, i.e., sequences $a_1<\dots<a_n$ such that $a_i = a_1+(i-1)e$ for some positive integer $e$ with $\gcd(a_1,e)=1$. For this family, we refine \cite[Cor.~4.2]{Sengupta2023} that  considers $a_1 > n-1$.

\begin{proposition}\label{pr:aritm}
Let $a_1 < \ldots < a_n=d$ be an arithmetic sequence of relatively prime integers, i.e., for all $i=1,\ldots,n$,
$a_i = a_1+(i-1)e$ for some integers $a_1,e>0$ such that $\gcd(a_1,e)=1$. Then, $(\AP_\mS, \leq_{\mS}) \simeq (\Ap_1, \leq_1)$ \iff{} $a_1 > n-2$. Therefore, if $a_1>n-2$, the Betti sequences of $k[\mC_1]$ and $k[\mC]$ coincide.
\end{proposition} 

\begin{proof}
We use Proposition~\ref{prop:APS_Ap1_isomorphic} to characterize when $(\AP_\mS,\leq_\mS)$ and $(\Ap_1,\leq_1)$ are isomorphic. When $a_1<\dots<a_n$ is an arithmetic sequence, $k[\mC]$ is always Cohen-Macaulay by  \cite[Cor.~2.3]{BGG}. Moreover, one can easily check that $\{a_1,\ldots,a_{n-1}\} \subset \MSG{\mS_1}$ if and only if $a_1 > n-2$.
Therefore, if $a_1 \leq n-2$, then $(\Ap_1,\leq_1)$ is not isomorphic to $(\AP_\mS,\leq_\mS)$.
Conversely, if $a_1>n-2$, it is sufficient to prove that $(\Ap_1,\leq_1)$ is graded.
By \cite[Thm.~3.4]{LPR}, the Apery set of $\mS_1$ is described as follows: if, for all $b\in \{0,\ldots,d-1\}$, $q_b$ and $-r_b$ denote respectively the quotient and the reminder of the division with negative remainder of $b$ by $n-1$, i.e., $q_b=\left\lceil b/(n-1)  \right\rceil$ and $r_b= q_b(n-1)-b$ with $0\leq r_b\leq n-2$, then
\[\Ap_1 = \left\{ q_b a_1 + r_be \,,\ 0\leq b\leq d-1 \right\} \,.\]
We claim that the grading is given by the function $\rho_1:\Ap_1 \rightarrow \N$ defined by $\rho_1\left( q_b a_1 + r_be \right)$ $ = q_b$. Consider $y, z\in \Ap_1$ such that $y \prec_1 z$, an let us prove that $\rho_1(z) = \rho_1(y)+1$. By Proposition~\ref{prop:CoverRel&RankFunct}~\ref{prop:CoverRel_a1}, there exist natural numbers $b\in \{0,\ldots,d-1\}$ and $i \in \{1,\ldots,n-1\}$ such that $y = q_b a_1 + r_b e$ and $z = y+a_i = (q_b+1)a_1 + (r_b+i-1)e$. If $i\geq n-r_b$, then $z-d = q_b a_1+\left( r_b+i-1-(n-1) \right)e \in \mS_1$, contradicting the fact $z\in \Ap_1$. Hence, $i \leq n-r_b-1$. Set $b' := (q_b+1)(n-1)-(r_b+i-1)$. As $0\leq r_b+i-1\leq n-2$, on the one hand one has that $0\leq b'\leq d-1$, on the other  $q_{b'} = q_b+1$ and $r_{b'} = r_b+i-1$. Therefore $z = q_{b'}a_1+r_{b'}e$, and hence $\rho_1(z)=\rho_1(y)+1$.
\end{proof}

\begin{example}
For the sequence $5<6<7<8<9<10$, one has that $a_1 = 5 > 4 = n-2$. Therefore, the Apery sets $(\Ap_1,\leq_1)$ and $(\AP_\mS,\leq_\mS)$ are isomorphic by Proposition~\ref{pr:aritm}. The Betti sequences of $k[\mC_1]$ and $k[\mC]$ coincide and one can check, using, e.g., \cite{Singular}, that both sequences are $(1,11,30,35,19,4)$. 
The isomorphic posets $(\Ap_1,\leq_1)$ and $(\AP_\mS,\leq_\mS)$ in this example are shown in Figure~\ref{fig:5678910}.

\begin{figure} 
\includegraphics[scale=.8]{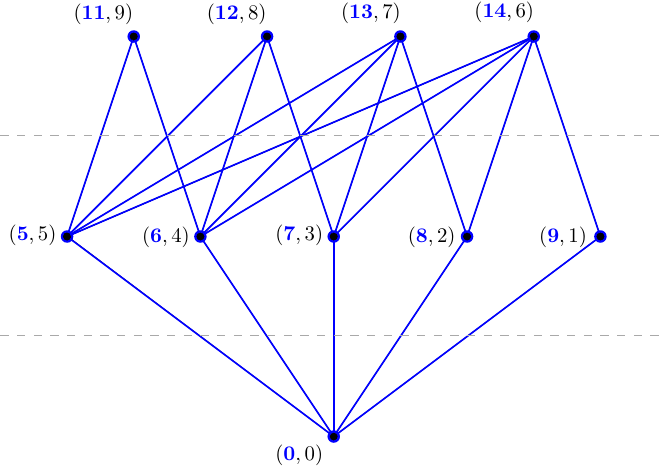}  
\caption{The posets $(\Ap_1,\leq_1)$ (in blue) and $(\AP_\mS,\leq_\mS)$ (in black) for $\mS_1 = \langle 5,6,7,8,9,10\rangle$.}
\label{fig:5678910}
\end{figure}
\end{example}

The next family that we now consider are monomial curves defined by an arithmetic sequence in which we have removed one term. 
In \cite[Sect. 6]{BGGM}, the authors study the canonical projections of the projective monomial curve $\mC$ defined by an
arithmetic sequence $a_1 < \cdots < a_{n}$ of relatively prime integers, i.e., the curve $\pi_r(\mC)$ obtained as the Zariski closure of the
image of $\mC$ under the $r$-th canonical projection $\pi_r: \mathbb P_k^{n} \dashrightarrow \mathbb P_k^{n-1}$, 
$(p_0:\cdots:p_{n}) \mapstochar\dashrightarrow (p_0: \cdots: p_{r-1} : p_{r+1} : \cdots : p_{n})$. We know that $\pi_r(\mC)$ is the projective 
monomial curve associated to the sequence $a_1 < \cdots < a_{r-1} < a_{r+1} < \cdots < a_{n}$.
\newline

If one removes either the first or the last term from an arithmetic sequence, the sequence is still arithmetic. Moreover, note that if an arithmetic sequence $a_1 < \cdots < a_{n}$ satisfies the condition $a_1>n-2$ in Proposition~\ref{pr:aritm}, then the arithmetic sequence obtained by removing either the first or the last term also satisfies the condition in Proposition~\ref{pr:aritm} because the number of terms in the new sequence is smaller, and its first term may have increased. Thus, we will only focus here on sequences obtained from an arithmetic sequence $a_1 < \cdots < a_{n}$ by removing $a_r$ for $r\in \{2,\ldots,n-1\}$.
Set $\mA_1 := \{a_1,\ldots,a_n\} \setminus \{a_r\}$, and consider the numerical semigroup $\mS_1=\mS_{\mA_1}$ and its homogenization $\mS$. We characterize in Proposition~\ref{pr:quasiaritm} when the posets $(\Ap_1,\leq_1)$ and $(\AP_\mS,\leq_\mS)$ are isomorphic. Two main ingredients in the proof are the following two results in \cite{BGGM} that we recall for convenience. The first one is a technical lemma, while the second describes the Apery set of $\mS_1$.

\begin{lemma}[{\cite[Lemma 2]{BGGM}}] \label{lemma:BGGM_2}
Let $a_1<\dots<a_n$ be an arithmetic sequence of relatively prime integers with common difference $e$. Set $q := \lfloor (a_1-1)/(n-1) \rfloor \in \N$ and $\ell := a_1-q(n-1) \in \{1,\ldots,n-1\}$. Then, 
\begin{enumerate}[(a)]
\item\label{lemma:BGGM_2a} $(q+e)a_1+a_i=a_{\ell+i}+qa_n$\,, for all $i\in \{1,\ldots,n-\ell\}$, and
\item\label{lemma:BGGM_2b} $q+e+1 = \min\{m>0 \mid m a_1 \in \langle a_2,\ldots,a_n\rangle\}$.
\end{enumerate}
\end{lemma}

\begin{lemma}[{\cite[Cor.~4]{BGGM}}] \label{lemma:Ap1_project}
Let $a_1<\dots<a_n$ be an arithmetic sequence of relatively prime integers with common difference $e$. Denote by $A$ the Apery set of $\bar{\mS}_1 = \langle a_1,\ldots,a_n \rangle$ with respect to $a_n$, $q:= \lfloor (a_1-1)/(n-1)\rfloor$, and, for all $\mu\in \N$, set $v_\mu :=\mu a_1+a_2$. Given $r\in \{2,\ldots,n-1\}$, consider $\mA_1 = \{a_1,\ldots,a_n\} \setminus \{a_r\}$, and the semigroup $\mS_1$ generated by $\mA_1$. When $a_1\geq r$, the Apery set of $\mS_1$ with respect to $a_n$ is described as follows:
\begin{enumerate}[(a)]
    \item\label{lemma:Ap1_project(2)} If $r=2$, 
\[\Ap_1 = \left\lbrace \begin{array}{ll}
    \left(A \setminus \{v_\mu \mid 0 \leq \mu \leq q+e\} \right) \cup \{v_\mu + a_n \mid 0 \leq \mu \leq q+e\}, & \text{if} \,\, n-1 \mid a_1,\\
    \left(A \setminus \{v_\mu \mid 0 \leq \mu \leq q+e-1\} \right) \cup \{v_\mu + a_n \mid 0 \leq \mu \leq q+e-1\}, & \text{otherwise}. 
    \end{array} \right.\]
    \item\label{lemma:Ap1_project(3-n2)} If $r\in \{3,\ldots,n-2\}$, \[\Ap_1 = \left( A \setminus \{a_r\} \right) \cup \{a_r+a_n\}.\]
    \item\label{lemma:Ap1_project(n1)} If $r = n-1$, \[\Ap_1 = \left\lbrace \begin{array}{ll}
        \left(A \setminus \{a_{n-1}\} \right) \cup \{a_{n-1}+(q+1)a_n\}, & \text{if} \,\, n-1 \mid a_1, \\
    \left(A \setminus \{a_{n-1}\} \right) \cup \{a_{n-1}+qa_n\}, & \text{otherwise}.
    \end{array} \right.\]
\end{enumerate}
\end{lemma}

\begin{proposition} \label{pr:quasiaritm}
Consider $a_1 < \ldots < a_{n}$ an arithmetic sequence of relatively prime integers with $n \geq 4$, and take $r\in \{2,\ldots,n-1\}$. Set $\mA_1 := \{a_1,\ldots,a_{n}\} \setminus \{a_r\}$, and let $\mS_1$ be the numerical semigroup generated by $\mA_1$, and $\mS$ its homogenization. Then,
\[(\AP_\mS, \leq_{\mS}) \simeq (\Ap_1, \leq_1)
\Longleftrightarrow
\begin{cases}
a_1> n-2 \text{ and } a_1\neq n,  & \text{if }r=2,\\
a_1\geq n \text{ and }r\leq a_1-n+1, & \text{if }3\leq r \leq n-2,\\
a_1 \geq n-2, & \text{if }r=n-1.
		 \end{cases}
\]
Consequently, if the previous condition holds, then $\beta_i(k[\mC_1]) = \beta_i(k[\mC])$, for all $i$.
\end{proposition}

\begin{proof}
Denote by $\bar{\mS}_1$ the numerical semigroup generated by the whole arithmetic sequence $a_1<\dots<a_n$. Again, we use Proposition~\ref{prop:APS_Ap1_isomorphic} to characterize when the posets $(\Ap_1,\leq_1)$ and $(\AP_\mS,\leq_\mS)$ are isomorphic. 
Note that $\{a_1,\ldots,a_n\}\setminus \{a_r\} \subset \MSG{\mS_1}$ if and only if 
\begin{equation}\label{eq:apmsg}
\hbox{either } r\neq n-1 \hbox{ and } a_1 > n-2,  \hbox{ or }  r=n-1  \hbox{ and }  a_1 \geq n-2.
\end{equation} 
On the other hand, by \cite[Cor.~5]{BGGM}, $k[\mC]$ is \cm{} if and only if 
\begin{equation}\label{eq:apCM}
r\leq a_1 \hbox{ or } r=n-1.
\end{equation}
Finally, by Proposition~\ref{prop:CoverRel&RankFunct}~\ref{prop:CoverRel_a2},  $(\Ap_1,\leq_1)$ is graded if and only if 
\begin{equation}\label{eq:aphomog}
\forall b \in \Ap_1,\ 
    b = \sum_{i \notin \{r,n\}} \alpha_r a_i = \sum_{i \notin \{r,n\}} \beta_i a_i \quad\Longrightarrow \sum_{i \notin \{r,n\}} \alpha_i = \sum_{i \notin \{r,n\}} \beta_i.
\end{equation}
We split the proof into three cases depending on the value of $r$.

\smallskip
$\bullet$ \ $r=2$.

\smallskip
By \eqref{eq:apmsg}, if $(\AP_\mS, \leq_{\mS}) \simeq (\Ap_1, \leq_1)$, then $a_1>n-2$.
If $a_1=n$, the element $a_3+a_{n-1}=a_2+a_n$ of $\Ap_1$ can be written as $(2+e)a_1$, and hence $(\Ap_1,\leq_1)$ is not graded by \eqref{eq:aphomog}. 
Assume now that $a_1> n-2$ and $a_1\neq n$, and let us prove that $(\Ap_1,\leq_1)$ is graded in this case. 
By Lemma~\ref{lemma:Ap1_project}~\ref{lemma:Ap1_project(2)}, 
\[ \Ap_1 = \left(A \setminus \{v_\mu \mid 0 \leq \mu \leq t\} \right) \cup \{v_\mu + a_n \mid 0 \leq \mu \leq t\},\]
for $t \in \{q+e-1,q+e\}$. 
Every $b \in A \cap \Ap_1$ satisfies (\ref{eq:aphomog}) by Proposition~\ref{pr:aritm}, so consider $b_\mu:= \mu a_1+a_2+a_n = \mu a_1+a_3+a_{n-1} \in \Ap_1$, with $0 \leq \mu \leq t$. Let us prove that whenever  $b_\mu = \sum_{i\notin \{2,n\}} \alpha_i a_i$, with $\alpha_i \in \N$, then $\sum_{i \notin \{2,n\}} \alpha_i = \mu+2$. 

Using iteratively the relations $a_i+a_j = a_{i-1}+a_{j+1}$ in $\bar{\mS}_1$, we get that
\[b_\mu = \sum_{i\notin \{2,n\}} \alpha_i a_i = \beta_1a_1+\epsilon a_m+\beta_n a_n\] 
for some $m$, $2 \leq m \leq n-1$, $\epsilon \in \{0,1\}$, and $\beta_1,\beta_n \in \N$ such that $\sum_{i\notin \{2,n\}} \alpha_i = \beta_1+\epsilon+\beta_n$. 

If $\epsilon=0$ or $m\neq 2$, then $a_2$ is not involved in the expression $b_\mu = \beta_1 a_1+\epsilon a_m+\beta_n a_n$, so $\beta_n=0$ since $b_\mu\in\Ap_1$. Thus, $b_\mu =\mu a_1+a_2+a_n = \beta_1 a_1+\epsilon a_m$, and hence
\begin{equation}\label{eq:proj_caso2}
(\beta_1-\mu) a_1= a_2+a_n-\epsilon a_m .
\end{equation}
If $\epsilon = 0$, $a_1$ divides $a_2+a_n=2a_1+ne$, and hence $a_1$ divides $n$ which is impossible since $a_1\geq n-1$ and $a_1\neq n$.
Now if $\epsilon = 1$ and $m\neq 2$, \eqref{eq:proj_caso2} implies that  
$(\beta_1-\mu) a_1= a_2+a_n-a_m=a_1+(n-m+1)e$,
and hence $a_1 \mid n-m+1$, a contradiction since $a_1\geq n-1>n-m+1$.
Thus, $\epsilon =1$ and $m=2$, i.e., 
$b_\mu= \sum_{i\notin \{2,n\}} \alpha_i a_i
= \mu a_1+a_2+a_n 
=\beta_1a_1+a_2+\beta_n a_n$.

Note that since $\beta_1a_1+a_2$ cannot be transformed into $\sum_{i\notin \{2,n\}} \alpha_i a_i$ using the relations $a_i+a_j = a_{i-1}+a_{j+1}$ in $\bar{\mS}_1$, we have that $\beta_n\neq 0$. Moreover, $(\mu-\beta_1)a_1 = (\beta_n-1)a_n$ and $\mu-\beta_1 < q+e+1$ since $\mu \leq t\leq q+e$. By Lemma~\ref{lemma:BGGM_2}~\ref{lemma:BGGM_2b}, this implies that $\mu=\beta_1$ and $\beta_n=1$, and we have shown that $\sum_{i\notin \{2,n\}} \alpha_i = \beta_1+\beta_n+1=\mu+2$.

\smallskip
$\bullet$ \ $3\leq r\leq n-2$.

\smallskip
By \eqref{eq:apmsg} and \eqref{eq:apCM}, the conditions $a_1 \geq n-1$ and $r\leq a_1$ are necessary for 
$(\AP_\mS, \leq_{\mS})$ and $(\Ap_1, \leq_1)$ to be isomorphic, and by Lemma~\ref{lemma:Ap1_project}~\ref{lemma:Ap1_project(3-n2)}, $\Ap_1 = \left( A \setminus \{a_r\} \right) \cup \{a_r+a_n\}$. Using Proposition~\ref{pr:aritm}, we get that $(\Ap_1,\leq_1)$ is graded if and only if every way of writing $a_r+a_n$ in terms of minimal generators of $\mS_1$ has the same number of summands, which is two since $a_r+a_n=a_{r+1}+a_{n-1}$. 

Now, if $a_1 = n-1$, then $a_{r+1}+a_{n-1} = ea_1+a_2+a_{r-1}$, and if $r>a_1-n+1$, then $ a_{r+1}+a_{n-1} = (2+e)a_1+(r-a_1+n-2)e = (1+e)a_1+a_{r-a_1+n-1}$. Thus, in both cases $(\Ap_1,\leq_1)$ is not graded.
Conversely, assume that $a_1 \geq n$ and $r \leq a_1 - n +1$. If $a_r+a_n=2a_1+(n+r-2)e$ can be written using more than $2$ minimal generators of $\mS_1$, then there exists 
$\mu\geq 3$ (the number of minimal generators involved), and
$m\geq 0$, such that $a_r+a_n=\mu a_1+me$. Then, $m\leq n+r-3$ and $a_1$ divides $n+r-2-m$, a contradiction since $a_1> n+r-2\geq n+r-2-m$.

\smallskip
$\bullet$ \ $r = n-1$.

\smallskip
By \eqref{eq:apmsg} and \eqref{eq:apCM}, we only have to show in this case that if $a_1\geq n-2$, then $(\Ap_1,\leq_1)$ is graded, i.e., using Lemma~\ref{lemma:Ap1_project}~\ref{lemma:Ap1_project(n1)} and Proposition~\ref{pr:aritm}, that \eqref{eq:aphomog} holds for $b=a_{n-1} + (q+1) a_{n}$ when $n-1\mid a_1$, and $b=a_{n-1} + q a_{n}$ otherwise.

Assume that $n-1$ does not divide $a_1$, and consider the element $b = a_{n-1} + q a_{n}$ in $\Ap_1$. By Lemma~\ref{lemma:BGGM_2}~\ref{lemma:BGGM_2a}, there exists $j\in \{1,\ldots,n-2\}$ such that $b = (q+e)a_1+a_j$, and hence we have to show that whenever $b = \sum_{i=1}^{n-2} \alpha_i a_i$ with $\alpha_i\in \N$, then $\sum_{i=1}^{n-2} \alpha_i = q+e+1$. As in the case $r=2$, using iteratively the equalities $a_i + a_j = a_{i-1} + a_{j+1}$ in $\bar{\mS}_1$, we get that 
\[b = \sum_{i=1}^{n-2} \alpha_i a_i = \beta_1a_1+\epsilon a_m+\beta_n a_n\] 
for some $m$, $2 \leq m \leq n-1$, $\epsilon \in \{0,1\}$, and $\beta_1,\beta_n \in \N$ such that $\sum_{i=1}^{n-2} \alpha_i = \beta_1+\epsilon+\beta_n$.

If $\beta_n > 0$, since $b \in \Ap_1$, we have that $b - a_n =  \beta_1 a_1 + \epsilon a_m + (\beta_n-1) a_n  \notin \mS_1$, and hence $\epsilon = 1$ and $m = n-1$, i.e., $b - a_n =  \beta_1 a_1 + a_{n-1} + (\beta_n-1) a_n $.
But this is also equal to $(\beta_1-1) a_1 + a_2 + a_{n-2} + (\beta_n-1) a_n$
so $\beta_1=0$ (otherwise $b - a_n\in \mS_1$). Thus, $b=a_{n-1}+\beta_n a_n$
that cannot be transformed into $\sum_{i=1}^{n-2} \alpha_i a_i$ using the relations $a_i + a_j = a_{i-1} + a_{j+1}$ in $\bar{\mS}_1$, a contradiction. This shows that $\beta_n=0$.

Then $b = \beta_1 a_1 + \epsilon a_m = (q+e) a_1 + a_j$. Since $\{a_1,\ldots,a_{n-2}\} \subset \MSG{S_1}$, we deduce that $\epsilon = 1$, $m = j$, and $\beta_1=q+e$. Hence, $\sum_{i = 1}^{n-2} \alpha_i = \beta_1 + \epsilon +\beta_n= q+e+1$, and we are done in this case.

When $n-1$ divides $a_1$, consider $b = a_{n-1} + (q+1) a_{n}$ in $\Ap_1$, and the relation $b = (q+e+1) a_1 + a_{n-1}$ given by Lemma~\ref{lemma:BGGM_2}~\ref{lemma:BGGM_2a}, and an analogue argument works. 
\end{proof}

\begin{example}
For the arithmetic sequence $9<10<11<12<13$, the parameters are $a_1 = 9$, $e=1$ and $n=5$. By Proposition~\ref{pr:aritm}, the Betti sequences of $k[\mC_1]$ and $k[\mC]$ coincide.  Indeed, it is $(1,10,20,15,4)$ for both curves. Now the Betti sequences of $k[\pi_r(\mC_1)]$ and $k[\pi_r(\mC)]$ also coincide for all values of $r$, $1\leq r\leq 5$: they coincide for $r=1$ and $5$ as observed before Lemma~\ref{lemma:BGGM_2}, and for $r=2,3,4$ by Proposition~\ref{pr:quasiaritm}. One can check that the sequence is 
$(1,6,8,3)$ for $r=1$, $(1,5,6,2)$ for $r=2$ and $4$, $(1,8,12,5)$ for $r=3$, and $(1,4,5,2)$ for $r=5$. 
\end{example}

\begin{example}\label{ex:quasiaritm2}
Consider the arithmetic sequence $9<10<11<12<13<14<15$, whose parameters are $a_1=9$, $e=1$ and $n=7$. By Proposition~\ref{pr:aritm}, the Betti sequences of $k[\mC_1]$ and $k[\mC]$ coincide, and one can check that it is $(1,19,58,75,44,11,2)$ for both the affine and the projective monomial curves. Now using the notations in Proposition~\ref{pr:quasiaritm}, one has that $\Ap_1$ and $\AP_\mS$ are isomorphic \iff{} $r \in \{2,3,6\}$. Hence, the Betti sequences of $k[\pi_r(\mC_1)]$ and $k[\pi_r(\mC)]$ coincide for those values of $r$ by Theorem~\ref{thm:BettiPosets}, and also for $r=1$ and $7$ as observed before Lemma~\ref{lemma:BGGM_2}. On the other hand, one can check using \cite{Singular} that the Betti sequences of $k[\pi_r(\mC_1)]$ and $k[\pi_r(\mC)]$ do not coincide for $r \in \{4,5\}$. Table \ref{tab:BettiSeq2} shows the Betti sequences of $k[\pi_r(\mC_1)]$ and $k[\pi_r(\mC)]$ for all $r$, $1\leq r \leq 7$.
\end{example}

\begin{table}[htbp]
\centering
\caption{Betti sequences in Example \ref{ex:quasiaritm2}.}
\label{tab:BettiSeq2}
\begin{tabular}{c@{\hspace{2.5em}}l@{\hspace{2.5em}}l}
\toprule
$r$ & $k[\pi_r(\mC_1)]$ & $k[\pi_r(\mC)]$ \\
\midrule
$1$ & $(1,11,30,35,19,4)$ & $(1,11,30,35,19,4)$ \\
$2$ & $(1,12,25,21,10,3)$ & $(1,12,25,21,10,3)$ \\
$3$ & $(1,13,30,29,14,3)$ & $(1,13,30,29,14,3)$ \\
$4$ & $(1,12,{\color{red}27,27},14,3)$ & $(1,12,{\color{red}29,29},14,3)$ \\
$5$ & $(1,{\color{red}12,25,21,10},3)$ & $(1,{\color{red}13,30,29,14},3)$ \\
$6$ & $(1,12,25,21,10,3)$ & $(1,12,25,21,10,3)$ \\
$7$ & $(1,12,25,25,14,3)$ & $(1,12,25,25,14,3)$ \\
\bottomrule
\end{tabular}
\end{table}

\section{Improving Vu's bound for equality of Betti numbers of a projective monomial curve and its projection}\label{sec:vu}

Take a sequence of nonnegative integers $0 = c_1 < \cdots < c_n$, not necessarily relatively prime, and consider, for all $j>0$, the shifted set of integers $\mA_1^j = \{c_1+j,\ldots,c_n+j\}$, and the semigroup $\mS_1^{j}$ generated by the sequence $a_0:=0<a_1:=c_1+j<\dots<a_n:=c_n+j$. Herzog and Srinivasan conjectured that the Betti numbers of $k[\mS_1^{j}]$ eventually become periodic with period $c_n$. In \cite{Vu}, Vu provides a proof of this conjecture together with an explicit value $N$ such  that this periodic behavior occurs for all $j > N$. One of the key steps in Vu's argument is \cite[Thm.~5.7]{Vu} where he proves that, for all $j > N$, the Betti numbers of the affine and projective monomial curves defined by $c_1 + j < \cdots < c_n + j$ coincide. We provide a smaller value of $N$ such that this occurs.

\begin{theorem}\label{thm:mejoravu} Let $0 = c_1 < \cdots < c_n$ be a sequence of nonnegative integers and set $N := (c_n - 1) (\sum_{i = 2}^{n-1} c_i)$. Then, for all $j \geq N$, the affine and projective monomial curves defined by the sequence $a_0=0 < a_1 = c_1 + j < \cdots < a_n = c_n + j$ have the same Betti numbers.
\end{theorem}

\begin{proof}
Take $j\geq N$. Let $\mG_>^j$ be the reduced Gr\"obner basis of $I_{\mA_1^j}$ with respect to the degrevlex order with $x_1 > \cdots > x_n$, and consider $f = \bx^{\alpha} - \bx^{\beta} \in \mG_>^j$ with $\bx^{\alpha} > \bx^{\beta}$. 
If we show that
\begin{enumerate}[(a)]
    \item\label{thm:mejoravu_a} $x_n$ does not divide $\bx^\alpha$, and
    \item\label{thm:mejoravu_b} if $f$ is not homogeneous, then $x_n$ divides $\bx^\beta$,
\end{enumerate}
then  the result follows from \cite[Thm.~4.1]{Sengupta2023}. Note that this result is true even if the generators of the semigroup are not relatively prime since the defining ideal does not change when we divide them by a common divisor.
\newline

If $x_n$ divides $\bx^\alpha$, then $x_n$ does not divide $\bx^\beta$, and hence $|\alpha| > |\beta|$. Thus,
\[ N=(c_n - 1) (\sum_{i = 2}^{n-1} c_i) \leq j \leq (|\alpha|-|\beta|)j<
\sum_{i = 1}^n (\alpha_i - \beta_i) j + \sum_{i = 2}^{n} \alpha_i c_i = \sum_{i = 2}^{n-1} \beta_i c_i \, . 
\] 
This implies that there exists $i \in \{2,\ldots,n-1\}$ such that $\beta_i \geq c_n$. If we consider the monomial  
$\bx^{\gamma}:=  \frac{\bx^{\beta} x_1^{c_n - c_i}  x_n^{c_i}}{x_i^{c_n}}$,
then the homogeneous binomial $g = \bx^{\beta} - \bx^{\gamma}$ belongs to 
$I_{\mA_1^j}$ because the homogeneous binomial $x_i^{c_n} - x_1^{c_n - c_i}  x_n^{c_i}$ belongs to $I_{\mA_1^j}$.  As $x_n$ divides $\bx^{\gamma}$ and does not divide $\bx^{\beta}$, $\ini_>(g) = \bx^{\beta} \in \ini_>(I_{\mA_1^j})$, a contradiction because $\mG_>^j$ is reduced and $f\in\mG_>^j$. This shows that that $x_n$ does not divide $\bx^\alpha$, and \ref{thm:mejoravu_a} is proved.
\newline

Now assume that $f$ is not homogeneous, i.e., $|\alpha| > |\beta|$, and that $x_n$ does not divide $\bx^\beta$. By \ref{thm:mejoravu_a}, $x_n$ does not divide $\bx^{\alpha}$ either, and hence
\[ N=(c_n - 1) (\sum_{i = 2}^{n-1} c_i) \leq j \leq (|\alpha|-|\beta|)j<
 \sum_{i = 1}^{n-1} (\alpha_i - \beta_i) j + \sum_{i = 2}^{n-1} \alpha_i c_i = \sum_{i = 2}^{n-1} \beta_i c_i \, . \]
Thus, there exists $i \in \{2,\ldots,n-1\}$ such that $\beta_i \geq c_n$. 
Using exactly the same argument as before for \ref{thm:mejoravu_a}, we get a contradiction, and hence \ref{thm:mejoravu_b} is proved.
\end{proof}

\begin{corollary}\label{cr:improveVu} Let $a_1 < \cdots < a_{n}$ be a sequence of positive integers, and set $M := a_n + (a_n -1)(\sum_{i = 1}^{n-1} (a_n - a_i))$. Then, for all $j \geq M$, the  projective monomial curve defined by the sequence $a_1 < \cdots < a_{n} < j$ is arithmetically Cohen-Macaulay.
\end{corollary}
\begin{proof}
Consider the sequence $b_0:=0 < b_1 := a_{n} - a_{n-1} < \cdots < b_{n-1} := a_{n} - a_1 < b_n := a_n$. By Theorem~\ref{thm:mejoravu}, one has that the projective monomial curve defined by $l<l+b_1 < \cdots < l+ b_n$ is arithmetically Cohen-Macaulay for all $l \geq (b_n - 1)(\sum_{i = 1}^{n-1} b_i) = (a_n-1)(\sum_{i = 1}^{n-1} (a_n-a_i))$. To finish the proof, it suffices to observe that the dual sequence of $0<a_1<\dots<a_n<l+a_n$ is $l<l+b_1<\dots<l+b_n$ and take $l+a_n=j$.
\end{proof}

\section{Construction of Gorenstein projective monomial curves}
\label{sec:gorenstein}

Since $\beta_i(k[\mC]) \geq \max\left( \beta_i(k[\mC_1]),\beta_i(k[\mC_2]) \right)$ for all $i$,
whenever $k[\mC]$ is Gorenstein, then so are $k[\mC_1]$ and $k[\mC_2]$. The converse of this statement is false; indeed, it could happen that $\mC_1$ and $\mC_2$ are both arithmetically Gorenstein and that $\mC$ is not even arithmetically Cohen-Macaulay, as can be seen in Example \ref{ex:Gorenstein}~\ref{ex:Gorenstein_a}. Actually, even if $k[\mC]$ is Cohen-Macaulay, it may happen that $k[\mC]$ is not Gorenstein, as Example \ref{ex:Gorenstein}~\ref{ex:Gorenstein_b} shows.

\begin{example}\label{ex:Gorenstein}
\begin{enumerate}[(a)]
    \item\label{ex:Gorenstein_a} The affine monomial curve $\mC_1$ defined by the sequence $4<9<10$ is an (ideal-theoretic) complete intersection and, thus, $k[\mC_1]$  is Gorenstein with Betti sequence $(1,2,1)$. The corresponding projective monomial curve is not arithmetically Cohen-Macaulay, indeed, the Betti sequence of $k[\mC]$ is $(1,5,6,2)$.
    \item\label{ex:Gorenstein_b} The affine monomial curve $\mC_1$ defined by the sequence $10<14<15<21$ is an (ideal-theoretic) complete intersection and, thus, $k[\mC_1]$ is Gorenstein with Betti sequence $(1,3,3,1)$. The corresponding projective monomial curve is arithmetically Cohen-Macaulay but not Gorenstein, indeed, the Betti sequence of $k[\mC]$ is $(1,4,5,2)$. 
\end{enumerate}
\end{example}

A numerical semigroup $\mS_1$ is symmetric if and only if either $b \in \mS_1$ or $F(\mS_1) - b \in \mS_1$ for all $b \in \N$, where $F(\mS_1) = \max \left( \N \setminus \mS_1 \right)$ is the Frobenius number of $\mS_1$. 
Kunz proved in \cite{Kunz} that $k[\mC_1]$ is Gorenstein if and only if $\mS_1$ is symmetric. In this section we show how to construct an arithmetically Gorenstein projective monomial curve from a symmetric numerical semigroup $\mT$. 
We begin with the following result, which provides a necessary and sufficient condition for $\mC$ to be arithmetically Gorenstein and is a consequence of the results in \cite{CN}.

\begin{proposition}\label{pr:caracterizaGorenstein}
Let $\mC$ be the projective monomial curve defined by the sequence $a_0=0 < a_1 < \cdots < a_n = d$ of relatively prime integers. Then, $\mC$ is arithmetically Gorenstein \iff{} $\mC$ is arithmetically Cohen-Macaulay, both $\mS_1$ and $\mS_2$ are symmetric, and $d$ divides $F(\mS_1) + F(\mS_2)$. 
\end{proposition}
\begin{proof} $(\Rightarrow)$ If $\mC$ is arithmetically Gorenstein, then $\mC$ is \acm{} and both $\mS_1$ and $\mS_2$ are symmetric by \cite[Lem.~2.6]{CN}. Assume now that $d$ does not divide $F(\mS_1) + F(\mS_2)$. By Proposition~\ref{prop:caractCM}~\ref{prop:caractCMc}, there exist $y \in \mS_1$ and $z \in \mS_2$ such that $(F(\mS_1) + d, z)$ and $(y,F(\mS_2)+d)$ are two different elements of $\AP_\mS$.  Moreover, by Proposition~\ref{pr:ordeninterseccion}, they are both maximal in the poset $(\AP_\mS,\leq_\mS)$, and hence,  $\mC$ is not arithmetically Gorenstein by \cite[Thm.~4.9]{CN}. \\
$(\Leftarrow)$ If $d$ divides $F(\mS_1) + F(\mS_2)$, then by Proposition~\ref{prop:caractCM}~\ref{prop:caractCMc}, $(F(\mS_1) + d, F(\mS_2) + d) \in \AP_\mS$ and by Proposition~\ref{pr:ordeninterseccion}, this element is the maximum of $(\AP_\mS,\leq_\mS)$. Hence, $\mC$ is arithmetically Gorenstein by \cite[Thm.~4.9]{CN}.
\end{proof}

Note that in the previous result, one cannot remove the condition of $k[\mC]$ being Cohen-Macaulay as Example \ref{ex:AfinGor_ProjNoCM} shows.

\begin{example} \label{ex:AfinGor_ProjNoCM}
For the sequence $6<7<8<15<16$, one has that the numerical semigroup $\mS_1 = \langle 6,7,8,15,16\rangle$ is symmetric, and $\mS_2 = \N$ is also symmetric. Moreover, $F(\mS_1) = 17$ and $F(\mS_2) = -1$, so $d=16$ divides $F(\mS_1)+F(\mS_2)$. But $k[\mC]$ is not Cohen-Macaulay, so it cannot be Gorenstein.
\end{example}

The following example provides a family of arithmetically Gorenstein projective curves. This example gives some insights on Theorem~\ref{th:gorenstein}, which is the main result of this section and shows how to construct a projective Gorenstein curve from a symmetric numerical semigroup.
For $a, b \in \mathbb Z$ with $a \leq b$, denote by $\llbracket a, b \rrbracket$ the integer interval $[a,b] \cap \Z$.

\begin{example} \label{rem:GorensteinFacil}
If  $m>3$ is an odd integer, one has that \[ \mS_1 = \left\langle (m+1)/2,\ldots,m-1 \right\rangle = \left\{0,(m+1)/2,\ldots,m-1,m+1,\rightarrow\right\}\] is a symmetric numerical semigroup with $F(\mS_1) = m$. Hence the ring $k[\mC_1]$ is Gorenstein. The sequence $\frac{m+1}{2}<\dots<m-1$ defines a projective curve of degree $d=m-1=F(\mS_1)-1$. We claim that $k[\mC]$ is Gorenstein. Note that 
$\Ap_1 = \{0\} \cup \llbracket \frac{m+1}{2},m-2 \rrbracket \cup \llbracket m+1, \frac{3}{2}(m-1) \rrbracket \cup \{2m-1\}$. Since $\mS_2 = \N$, we have that $F(\mS_2) = -1$  and $\Ap_2 = \llbracket 0,m-1 \rrbracket$.
By Proposition~\ref{pr:caracterizaGorenstein}, it only remains to check that $k[\mC]$ is Cohen-Macaulay. By Proposition~\ref{prop:caractCM}~\ref{prop:caractCMd}, $k[\mC]$ is Cohen-Macaulay because $B\subset \mS$, where $B \subset \N^2$ is the following set with $d$ elements: 
\[\{ (0,0)\} \cup \{(a,d-a) \mid \tfrac{m+1}{2}\leq a \leq m-2\} \cup \{ (d+g,d-g) \mid 1<g<\tfrac{m+1}{2} \} \cup \{(2d+1,d-1)\} \, .\]
\end{example}

We now generalize this to any symmetric numerical semigroup $\mT$ such that $\mT \neq \N$ and $\mT \neq \langle 2, a \rangle$ for some $a$ odd or, in other words, such that $2 \notin \mT$.
The idea under this construction is to consider the projective closure of the affine monomial curve parametrized by the so-called {\it small elements} in the semigroup, that is, all the elements in the numerical semigroup that are smaller than the Frobenius number. 
The precise statement of the result is the following.

\begin{theorem}\label{th:gorenstein} Let $\mT \subseteq \N$ be a symmetric numerical semigroup such that $2 \notin \mT$ and consider $\mT \cap \llbracket 0, F(\mT) - 1 \rrbracket = \{0,a_1,\ldots,a_n\}$ with $0<a_1 < \cdots < a_n$.
Then, the projective monomial curve defined by the sequence $a_1 < \cdots < a_n$ is arithmetically Gorenstein.
\end{theorem}

To prove this theorem we use the following two lemmas. We believe that they are known but since we could not find a precise reference, we decided to include their proofs.

\begin{lemma}\label{pr:frobsim}
Let $\mT \subset \N$ be a symmetric numerical semigroup and consider $a_1 <\cdots < a_n$ its minimal set of generators. If $2 \notin \mT$, then $a_n < F(\mT)$. 
\end{lemma}

\begin{proof}
We prove that every $y \in \mT$ such that $y > F(\mT)$ can be written as $y = z_1 + z_2$ with $z_1,z_2 \in \mT \setminus \{0\}$ and, hence, $y\notin \MSG{\mT}$.
\begin{itemize}
    \item For $y = F(\mT) + 1$, we take $z_1 = a_1 \in \mT$ and $z_2 = F(\mT) - a_1 + 1$. We have that $z_2 \in \mT$ because $F(\mT) - z_2 = a_1 - 1 \notin \mT$ and $\mT$ is symmetric.
    \item For $y = F(\mT) + 2$, we take $z_1 = a_1 \in \mT$ and $z_2 = F(\mT) - a_1 + 2$. We have that $z_2 \in \mT$ because $F(\mT) - z_2 = a_1 - 2 \notin \mT$ (because $a_1>2$) and $\mT$ is symmetric.
    \item For $y = F(\mT) + 3$. If $y/2 \in \mT$, we take $z_1 = z_2 = y/2$. Otherwise, we observe that \[ |\llbracket 1,y-1 \rrbracket \cap \mT| = |\llbracket 1,F(\mT) \rrbracket \cap \mT| + y - F(\mT) - 1 = y - \frac{F(\mT)+3}{2} = \frac{y}{2} \, . \] Thus, there exists $1 \leq i < y/2$ such that $i,y-i \in \mT$ and we are done. 
    \item For $y > F(\mT) + 3$, we observe that \[ |\llbracket 1,y-1 \rrbracket \cap \mT| = |\llbracket 1,F(\mT) \rrbracket \cap \mT| + y - F(\mT) - 1 = y - \frac{F(\mT)+3}{2} > \frac{y}{2} \, . \]
Thus there exists $1 \leq i \leq y/2$ such that $i,y-i \in \mT$ and we are done. 
\end{itemize}
\end{proof}

\begin{lemma}\label{lm:soloalprincipio} 
Let $\mS_1 = \langle a_1,\ldots,a_n \rangle \subsetneq \N$ be a numerical semigroup with $a_1 < \cdots < a_n$, and set $a := \min \{b \in \mS_1 : a_1 \nmid b \}$.   If $y \in \N$ satisfies that $y + i \notin \mS $ for all $i \in \{0,\ldots,a-1\}$ such that $a_1 \nmid i$, then $y = 0$.   
\end{lemma}

\begin{proof} Since $y+1,\ldots,y+a_1-1 \notin \mS_1$, we deduce that $a_1$ divides $y$, so $y\in \mS_1$. Moreover, $a-a_1$ is not a multiple of $a_1$, so $y+a-a_1 \notin \mS_1$ and $y + a - a_1 \equiv a \pmod{a_1}$. Thus, we get that $y + a - a_1 \leq a - a_1$, and hence $y = 0$.
\end{proof}

\begin{proof}[Proof of Theorem~\ref{th:gorenstein}]
Since $\mT$ is symmetric and $2 \notin \mT$, then by Lemma~\ref{pr:frobsim} we have that $\MSG{\mT} \subset \{a_1,\ldots,a_n\}$. Hence, $\mS_1 = \mT$ and $\mS_1$ is symmetric.  Moreover, since $1,2 \notin \mS_1$, then $d = a_n = F(\mS_1) - 1$ and $a_{n-1} = F(\mS_1) - 2$. Thus $\mS_2 = \N$ and we get that $F(\mS_2) = -1$ and $F(\mS_1) + F(\mS_2) = d$. By Proposition~\ref{pr:caracterizaGorenstein}, it is enough to prove that $\mC$ is \acm{} to conclude that it is arithmetically Gorenstein.

One can easily check that $\Ap_1 = \{a \in \mS_1 \, \vert \, 0 \leq a < d\} \cup \{g+d \, \vert \, g \notin \mS_1,  \, 1 < g < d\} \cup \{2d+1\}$, and $\Ap_2 = \{0,1,\ldots,d-1\}$. Consider now the following set $B \subset \N^2$ with $d$ elements: \[ \begin{split}
B = &\{ (0,0)\} \cup  \{(a,d-a)  \, \vert \, a \in \mS_1, 1 < a < d\} \cup \\ 
&\{ (d+g,d-g) \, \vert \, g \notin \mS_1,  1 < g < d\}  \cup   \{ (2d+1,d-1)\} \, .    
\end{split}\]
By Proposition~\ref{prop:caractCM}~\ref{prop:caractCMd}, $\mC$ is arithmetically Cohen-Macaulay \iff{} $B \subset \mS$, and in this case $\AP_\mS = B$. Let us prove that $B\subset \mS$. Clearly $(0,0) \in \mS$ and $\{(a,d-a)\, \vert \, a \in \mS_1,\, 0 < a < d\} = \{(a_i,d-a_i) \, \vert \, 1 \leq i < n\} \subset \mS$, and one has to show that $(d+g,d-g) \in \mS$ for all $g \notin \mS_1, 1 < g < d$ and $(2d+1,d-1) \in \mS$. 
Let $a \in \mS_1$ be the minimum element in $\mS_1$ which is not a multiple of $a_1$. We distinguish between two cases.
\begin{itemize}
\item[\underline{Case 1}] $d > g > F(\mS_1) - a = d + 1 - a$. We claim that $g+1 \in \mS_1$. Otherwise, by the symmetry of $\mS_1$ one has that $F(\mS_1) - g$ and $F(\mS_1) - g - 1$ are two consecutive elements of $\mS_1$ which are both smaller than $a$, and this is not possible. Then, $(d-1,1),\, (g+1,d-g-1) \in \mS$ and we get that $(d+g,d-g) = (d-1,1) + (g+1,d-g-1) \in \mS$. \\
\item[\underline{Case 2}] $1 < g \leq F(\mS_1) - a = d + 1 - a$. We claim that there exists $j \in \{0,\ldots,a-1\}$ such that both $d+1-j$ and $g-1+j$ belong to $\mS_1$. Assume by contradiction that this statement does not hold. Whenever $j \in \{0,\ldots,a-1\}$ is not a multiple of $a_1$, we have that $j \notin \mS_1$ and, by the symmetry of $\mS_1$, 
$F(\mS_1) - j = d+1-j \in \mS_1$ and hence $g-1+j \notin \mS_1$. By Lemma~\ref{lm:soloalprincipio}, this means that $g = 1$, a contradiction. Now, we take $j \in \{0,\ldots,a-1\}$ such that $d+1-j,g-1+j \in \mS_1$ (clearly $j \neq 0$ because $d+1 = F(\mS_1) \notin \mS_1$). Then $(g-1+j,d+1-g-j), (d+1-j,j-1) \in \mS$ and hence $(d+g, d-g) = (g-1+j,d+1-g-j) + (d+1-j,j-1) \in \mS.$
\end{itemize}
Finally, taking any $ g \notin \mS_1$, $1 < g < d$, we have that $F(\mS_1) - g = d+1 - g \in \mS_1$. Thus, $(2d+1,d-1) = (d+g,d-g) + (d+1-g,g-1) \in \mS.$ 
\end{proof}

\begin{remark}
By the proof of Theorem~\ref{th:gorenstein} and \cite[Thm.~3.6]{GG23}, it follows that the Castelnuovo-Mumford regularity of $k[\mC]$ is $\reg{k[\mC]} = 3$ for all the Gorenstein curves that we constructed in Example \ref{rem:GorensteinFacil} and Theorem~\ref{th:gorenstein}.
\end{remark}

Following the construction in Theorem~\ref{th:gorenstein}, one gets an arithmetically Gorenstein projective curve $\mC$. However, the Betti numbers of $k[\mC_1]$ and $k[\mC]$ can be very different, as the following example shows.

\begin{example}\label{ex:ejGorenstein}
Consider the symmetric numerical semigroup $\mT = \langle 4, 9, 10 \rangle$. One has that the Frobenius number of $\mT$ is $F(\mT) = 15$ and, hence, $\mT \cap \llbracket 0,14 \rrbracket = \{0,4,8,9,10,12,$ $13,14\}$. By Theorem~\ref{th:gorenstein} we have that the projective monomial curve defined by the sequence $4 < 8 < 9 < 10 < 12 < 13 < 14$ is Gorenstein. A computation with \cite{Singular} shows that the Betti sequence of $k[\mC_1]$ is $(1,6,15,20,15,6,1)$, while the Betti sequence of $k[\mC]$ is $(1,15,39,50,39,15,1)$. 
\end{example}

\section{Conclusions / Open questions}

Let $\mC_1 \subseteq \A_k^n $ be an affine monomial curve and consider $\mC \subseteq \P_k^n$ its projective closure. The Betti numbers of $k[\mC]$ (the coordinate ring of $\mC$) are always greater than or equal to those of $k[\mC_1]$ (the coordinate ring of $\mC_1$). We explore when the Betti numbers of both rings are identical. In Theorem~\ref{thm:BettiPosets}, which is our main result, we provide a sufficient condition for this in terms of the poset structure of certain Apery sets. As Example \ref{ex:convthm2.1}  shows, this condition is not necessary. The following natural question remains open.

\begin{prob}
Characterize when the Betti numbers of the coordinate rings of $\mC_1$ and $\mC$ coincide.
\end{prob}

Proposition~\ref{prop:Ap1graded_Sengupta} is an intermediate key result to obtain our main one. It translates a Gr\"obner-basis condition on the ideal of $\mC_1$ into a poset flavored combinatorial criterion. In one of the implications we added the hypothesis that $k[\mC]$ is Cohen-Macaulay and our proof relies on this. We wonder if this assumption can be dropped.

\medskip

We applied Theorem~\ref{thm:BettiPosets} to provide families of affine curves whose coordinate rings have the same Betti numbers as their corresponding homogenizations.  Also, we used our results to study the shifted family of monomial curves, i.e., the family of curves associated to the sequences $j+a_1 < \cdots < j+a_n$ for different values of $j \in \mathbb N$. In this context, Vu proved in \cite{Vu} that for $j$ big enough, one has an equality between the Betti numbers of the corresponding affine and projective curves. Using our results, we improve  Vu's threshold in Corollary~\ref{cr:improveVu}. Vu also showed that the Betti numbers become periodic in the shifted family for $j > N$, being $N$ an explicit value. Computational experiments suggest that the value of $N$ can be optimized. Hence, we propose the following problem.

\begin{prob}
Improve Vu's bound on the least value of $j$ such that the Betti numbers of the shifted family become periodic.
\end{prob}

A partial solution to this problem can be found in \cite{Dumishift}, where the author finds a better estimate when $\mS_1$ is a three generated numerical semigroup.

\medskip

Consider the projective monomial curve $\mC \subset \P_k^n$ given parametrically by $x_i = u^{a_i}v^{d-a_i}$ for all $i \in \{0,\ldots,n\}$, with $0 = a_0 < a_1 < \ldots < a_n = d$ a sequence of relatively prime integers. One can associate to $\mC$ the numerical semigroup $\mS_{\mC} = \langle a_1,\ldots,a_n \rangle$. It is known that whenever $\mC$ is arithmetically Gorenstein, then $\mS_{\mC}$ is symmetric. In the fourth section we prove a sort of converse to this statement. More precisely, in Theorem~\ref{th:gorenstein}, for a symmetric numerical semigroup $\mT$ we construct an arithmetically Gorenstein projective monomial curve $\mC$ such that $\mT = \mS_{\mC}$. In view of our result, we wonder if the answer to the following question could be positive:

\begin{prob}
For a numerical semigroup $\mT$ of type $t > 1$, does there exist an arithmetically Cohen-Macaulay curve $\mC$ of type $t$ such that $\mS_{\mC} = \mT$? 
\end{prob}

\medskip

\textbf{Conflict of Interest.} On behalf of all authors, the corresponding author states that there is no conflict of interest.


\begin{thebibliography}{99}

\bibitem{BGG}{
I. Bermejo, E. Garc\'{\i}a-Llorente, and I. Garc\'{\i}a-Marco,
Algebraic invariants of projective monomial curves associated to generalized arithmetic sequences,
{\em J. Symb. Comput.} {\bf 81} (2017), 1--19.}

\bibitem{BGGM} {
I. Bermejo, E. Garc\'ia-Llorente, I. Garc\'ia-Marco, and M. Morales. Noether resolutions in dimension 2. {\em J. Algebra} {\bf 482} (2017), 398--426.}

\bibitem{BGM} {
I. Bermejo and I. García-Marco. Complete intersections in simplicial toric varieties. {\em J. Symb. Comput.} {\bf 68} (2015), 265--286. }



\bibitem{CG}{
A. Campillo and P. Gimenez,
Syzygies of affine toric varieties,
{\em J. Algebra} {\bf 225} (2000), 142--161.}

\bibitem{CN}{
M. P. Cavaliere and G. Niesi,
On monomial curves and Cohen-Macaulay type,
{\em Manuscripta Math.} {\bf 42} (1983), 147--159.}


\bibitem{Singular} {
W. Decker, G.-M. Greuel, G. Pfister, and H. Sch\"{o}nemann. Singular 4-3-0 — A computer algebra system for polynomial computations. http://www.singular.uni-kl.de, 2022.
}


\bibitem{GG23}{
P. Gimenez and M. Gonz\'alez S\'anchez. Castelnuovo-Mumford regularity of projective monomial curves via sumsets, {\em Mediterr. J. Math.} {\bf 20}, 287 (2023), 24 pp.}


\bibitem{GSS}{
P. Gimenez, I. Sengupta, and H. Srinivasan.
Minimal graded free resolutions for monomial curves defined
by arithmetic sequences,
{\em J. Algebra} {\bf 388} (2013), 294--310.}


\bibitem{EjShibuta}{S. Goto, W. Heinzer, and M.-K. Kim. The leading ideal of a complete intersection of height two, Part II,
\emph{J. Algebra}, {\bf 312} (2007), 709--732.}


\bibitem{GSW}{
S. Goto, N. Suzuki, and K. Watanabe.
On affine semigroup rings,
{\em Jap. J. Math.} {\bf 2} (1976), 1--12.}

\bibitem{HerzogStamate2019}{
B. Herzog and D. I. Stamate. Cohen-Macaulay criteria for projective monomial curves via Gr\"obner bases. {\em Acta Math. Vietnam} {\bf 44} (2019), 51--64.
}


\bibitem{JZ} R. Jafari and S. Zarzuela Armengou. Homogeneous numerical semigroups. {\em Semigr. Forum} {\bf  97} (2018), 278--306.


\bibitem{Kunz} E. Kunz. The Value-Semigroup of a One-Dimensional Gorenstein Ring. {\em Proc. Amer. Math. Soc.} {\bf 25} (1970), 748--751.

\bibitem{LPR} {P. Li, D. E. Patil, and L. G. Roberts. Bases and ideal generators for projective monomial curves, {\em Commun. Algebr.} {\bf 40}
(2012), 173--191.}

\bibitem{Sengupta2023}{
J. Saha, I. Sengupta, and P. Srivastava. Betti sequence of the projective closure of affine monomial curves. {\em J. Symb. Comput.} {\bf 119} (2023), 101-111.
}

\bibitem{Dumishift} D. I. Stamate. Asymptotic properties in the shifted family of a numerical semigroup with few generators. {\em Semigr. Forum} {\bf 93} (2016), 225-246.

\bibitem{Stamate} {D. I. Stamate. Betti Numbers for Numerical Semigroup Rings. In: Ene V., Miller E. (eds) Multigraded Algebra and Applications. NSA 2016. Springer Proceedings in Mathematics \& Statistics, vol 238.}

\bibitem{Villa} {
R. H. Villarreal.
Monomial algebras.
Monographs and research notes in mathematics, 2nd edn. CRC
Press, Boca Raton (2015).}


\bibitem{Vu}{
T. Vu. Periodicity of Betti numbers of monomial curves. {\em J. Algebra} {\bf 418} (2014), 66--90.
}

\end{thebibliography}
\end{document}